\documentclass[11pt,a4paper]{article}
\usepackage{amsmath, amsfonts,amsthm,amssymb,amsbsy,upref,color,graphicx,amscd,hyperref,makeidx,enumerate}
\usepackage[active]{srcltx}
\usepackage[latin1]{inputenc}

\DeclareMathOperator{\dist}{dist}

\def\eps{{\varepsilon}}
\def\1{\mathbf{1}}
\def\dim{d}
\def\N{\mathbb{N}}

\def\R{\mathbb{R}}

\def\H{\mathcal{H}}

\def\Dr{D}

\newcommand{\be}{\begin{equation}}
\newcommand{\ee}{\end{equation}}
\newcommand{\bib}[4]{\bibitem{#1}{\sc#2: }{\it#3. }{#4.}}

\newcommand{\cp}{\mathop{\rm cap}\nolimits}

\numberwithin{equation}{section}
\theoremstyle{plain}
\newtheorem{teo}{Theorem}[section]
\newtheorem{lemma}[teo]{Lemma}

\newtheorem{prop}[teo]{Proposition}
\newtheorem{deff}[teo]{Definition}
\theoremstyle{remark}
\newtheorem{oss}[teo]{Remark}

\title{Existence and regularity of minimizers for some spectral functionals with perimeter constraint}

\author{Guido De Philippis, Bozhidar Velichkov}

\date{}

\begin{document}

\maketitle

\begin{abstract}
In this paper we prove that the shape optimization problem 
$$\min\left\{\lambda_k(\Omega):\ \Omega\subset\R^d,\ \Omega\ \hbox{open},\ P(\Omega)=1,\ |\Omega|<+\infty\right\},$$
where $\lambda_k(\Omega)$ is the $k$th eigenvalue of the Dirichlet Laplacian,
has a solution for any $k\in\N$ and dimension $d$. Moreover, every solution is a bounded connected open set with boundary which is $C^{1,\alpha}$ outside a closed set of Hausdorff dimension $d-8$. Our results apply to general spectral functionals of the form $f(\lambda_{k_1}(\Omega),\dots,\lambda_{k_p}(\Omega))$, for increasing functions $f$ satisfying some suitable bi-Lipschitz type condition. 
\end{abstract}

\vspace{1cm}
\section{Introduction}
A shape optimization problem is a variational problem of the type
\begin{equation}\label{F}
\min\left\{F(\Omega):\ \Omega\in\mathcal{A}\right\},
\end{equation}
where $\mathcal{A}$ is an \emph{admissible family} (or \emph{admissible set}) of domains in $\R^d$ and $F$ is a given \emph{cost functional}. For a detailed introduction to this type of problems we refer to the books \cite{bubu05,hen06,hepi05} and the papers \cite{but10,hen03}. 

\medskip
   A particularly interesting and widely studied class of shape optimization problems is the one in which the admissible set $\mathcal{A}$ is composed of open domains with perimeter or measure constraint and the cost functional $F$ depends on the solution of some partial differential equation on $\Omega$. For example, $F(\Omega)=\lambda_k(\Omega)$, where $\lambda_k(\Omega)$ is the $k$th eigenvalue of the Dirichlet Laplacian, i.e. the \(k\)th smallest positive real number such that the equation 
$$-\Delta u_k=\lambda_k(\Omega) u_k,\qquad u_k\in H^1_0(\Omega),$$
has a non-trivial solution. Another typical example is 
\[
F(\Omega)=E_\rho(\Omega):=-\frac12\int_\Omega \rho w\,dx,
\]
 where $\rho\in L^2(\R^d)$ and $w$ is the solution of
$$-\Delta w=\rho,\qquad w\in H^1_0(\Omega).$$
Unfortunately, there is no available technique to directly prove the existence of a solution of \eqref{F} in the family of open sets. The standard approach is to extend the cost functional to a wider class of domains and then prove that the minimizers are open sets, hopefully with a smooth boundary. 

   The existence in the case when the domains are measurable sets is well-known if the admissible family of domains consists of sets which satisfy some geometric constraint which assures a priori some compactness. Under some monotonicity and semi-continuity assumptions on the cost functional $F$ a general theorem of Buttazzo and Dal Maso states that there are solutions of \eqref{F}, in each of the cases 
\begin{equation}\label{mc}
\mathcal{A}=\left\{\Omega\subset\Dr:\ \Omega\ \hbox{Lebesgue measurable},\ |\Omega|\le c\right\},
\end{equation}
(see \cite{budm91,budm93}) and  
\begin{equation}\label{pc}
\mathcal{A}=\left\{\Omega\subset\Dr:\ \Omega\ \hbox{Lebesgue measurable},\ P(\Omega)\le c,\ |\Omega|<+\infty \right\},
\end{equation} 
(see \cite{bubuhe}). Here $|\Omega|$ is  the Lebesgue measure of \(\Omega\), $P(\Omega)$  the De Giorgi perimeter of \(\Omega\) (see \cite{giusti,maggi}) and the \emph{design region} $\Dr$ is a \emph{bounded} open set in $\R^d$. Clearly, in each of these cases the functional $F$ has to be appropriately extended on the class of Lebesgue measurable sets in $\Dr$ (see Section \ref{pr} for more details). 

\bigskip

The case $\Dr=\R^d$ is more involved and it was only recently
proved (see  \cite{bulbk} and  \cite{mp}) that when $F(\Omega)=\lambda_k(\Omega)$, for some $k\in\N$, there exists a solution of the problem \eqref{F}, where the admissible set $\mathcal{A}$ is given by \eqref{mc}. 

\medskip
   Once the existence of a solution of \eqref{F} is established in the class of measurable sets it is quite natural to ask whether these optimal sets are open (hence being also  a solution of the same problem in the more ``natural'' class of  open sets) and, in case the answer is affirmative, which is the regularity of the boundary of these optimal sets. 
  
   The study of the regularity of the optimal set in the case of a measure constraint strongly depends on the nature of the cost functional. Even the openness is a difficult question which is known to have a positive answer in the case of the Dirichlet energy $F(\Omega)=E_\rho(\Omega)$ (see \cite{brhapi}) and some spectral functionals $F(\Omega)=f(\lambda_1(\Omega),\dots,\lambda_k(\Omega))$ (see \cite{brla}, \cite{bucve} and \cite{bmpv}). 

   In the cases $F(\Omega)=E_\rho(\Omega)$, $F(\Omega)=\lambda_1(\Omega)$ the problem \eqref{F} with admissible set \eqref{mc} can be written in terms of a single function in $H^1_0(\Dr)$, i.e. the corresponding shape optimization problems are equivalent to
\begin{equation}\label{envar}
\min\left\{\frac12\int_{\Dr} |\nabla w|^2\,dx-\int_{\Dr} \rho w\,dx:\ w\in H^1_0(\Dr),\ |\{w\neq 0\}|\le c\right\},
\end{equation}
and
\begin{equation}\label{lbvar}
\min\left\{\frac{\int_{\Dr} |\nabla u|^2\,dx}{\int_{\Dr} u^2\,dx}:\ u\in H^1_0(\Dr),\ |\{u\neq 0\}|\le c\right\},
\end{equation}
respectively. For these functionals one can apply the techniques introduced in \cite{altcaf} to study the regularity of the optimal domains (see \cite{brhapi} and \cite{brla}, for two different arguments based on this idea). When the cost functional $F(\Omega)$ depends on the spectrum of the Dirichlet Laplacian the regularity of the optimal sets is still not known in the case of a measure constraint. 
  The main difficulty in this case is due to the fact that the variational formulation of $\lambda_k(\Omega)$ does not concern functions but $k$-dimensional spaces of functions, which makes the analysis quite difficult. More precisely, we have
\begin{equation}\label{lambdak}
\lambda_k(\Omega)=\min_{K\subset H^1_0(\Omega)}\,\max_{u\in K}\,\,\frac{\int_{\Omega} |\nabla u|^2\,dx}{\int_{\Omega} u^2\,dx},
\end{equation}
where the minimum is over all $k$-dimensional subspaces $K$ of $H^1_0(\Omega)$. 

   In this paper we are interested in the study of  existence and regularity for the problem \eqref{F} under the perimeter constraint \eqref{pc} when the design region is $\Dr=\R^d$. Our main result is the  following:
\begin{teo}\label{thkperint}
The shape optimization problem 
\begin{equation}\label{lbko}
\min\left\{\lambda_k(\Omega):\ \Omega\subset\R^d,\ \Omega\ \hbox{open},\ P(\Omega)=1,\ |\Omega|<\infty\right\},
\end{equation}
where $\lambda_k(\Omega)$ is defined in \eqref{lambdak} through the classical Sobolev space $H^1_0(\Omega)$ on the open set $\Omega$, has a solution. Moreover, any optimal set $\Omega$ is bounded  and connected. The boundary $\partial\Omega$ is $C^{1,\alpha}$, for every $\alpha\in(0,1)$, outside a closed set of Hausdorff dimension at most $d-8$.
\end{teo}   

This result is a consequence of the more general Theorem \ref{thfperint} which applies to spectral functionals of the form $$F(\Omega)=f(\lambda_{k_1}(\Omega),\dots,\lambda_{k_p}(\Omega)),$$ 
where $f:\R^p\to\R$ is an increasing locally Lipschitz function satisfying some local bi-Lipschitz-type assumption. More precisely, we consider $f$ such that:
\begin{enumerate}[($f$1)]
\item $f(x)\to+\infty$ as $|x|\to+\infty$; 
\item $f$ is locally Lipschitz continuous;
\item $f$ is increasing, i.e. for any $x=(x_1,\dots,x_p)\in\R^p$ and $y=(y_1,\dots,y_p)\in\R^p$ such that $x\ge y$, i.e. satisfying $x_j\ge y_j$, for every $j=1,\dots,p$, we have $f(x)\ge f(y)$. More precisely we assume that  for every compact set $K\subset\R^d\setminus\{0\}$, there exists a constant $a>0$ such that for any $x,y\in K$, $x\ge y$,  $$f(x)-f(y)\ge a|x-y|.$$ 
\end{enumerate}
For example, any polynomial of $\lambda_{k_1},\dots,\lambda_{k_p}$ with positive coefficients satisfies $(f1)$, $(f2)$ and $(f3)$.    

\begin{teo}\label{thfperint}
Suppose that $f:\R^p\to\R$ satisfies the assumptions $(f1)$, $(f2)$ and $(f3)$. Then the shape optimization problem 
\begin{equation}\label{lbfperint}
\min\left\{f\left(\lambda_{k_1}(\Omega),\dots,\lambda_{k_p}(\Omega)\right):\ \Omega\subset\R^d,\ \Omega\ \hbox{open},\ P(\Omega)=1,\ |\Omega|<\infty\right\},
\end{equation}
has a solution. Moreover, any optimal set $\Omega$ is bounded and connected and its boundary $\partial\Omega$ is $C^{1,\alpha}$, for every $\alpha\in(0,1)$, outside a closed set of Hausdorff dimension at most $d-8$.
\end{teo} 
\begin{proof}
The existence of an optimal set is first proved in the class of measurable sets in Theorem \ref{thfper}. The existence of a solution $\Omega$ of \eqref{lbfperint} is proved in Theorem \ref{thfpero}. The boundedness of $\Omega$ is due to the facts that $\Omega$ is an energy subsolution (see Definition \ref{sub} and Proposition \ref{lag}) and that every energy subsolution is a bounded set (Lemma \ref{bounded}). The problem of regularity of the optimal set is treated in Section \ref{reg}. More precisely, by Proposition \ref{lag} and Lemma \ref{super}, we have that $\Omega$ is an energy subsolution and a perimeter supersolution. In Theorem \ref{enper}, we prove that any $\Omega$ satisfying those two conditions has $C^{1,\alpha}$ boundary, for every $\alpha\in(0,1)$, outside a closed set of Hausdorff dimension at most $d-8$. The connectedness of the optimal set follows by Proposition \ref{conn}.   
\end{proof}
%

\begin{oss}
The regularity of the free boundary proved in Theorems \ref{thkperint} and \ref{thfperint} is not in general optimal. Indeed, it was shown in \cite{bubuhe} that the solution $\Omega$ of \eqref{lbko} for $k=2$ has smooth boundary. The proof is based on a perturbation technique and the fact that $\lambda_2(\Omega)>\lambda_1(\Omega)$ and can be applied for every $k\in\N$ under the assumption that the optimal set is such that $\lambda_k(\Omega)>\lambda_{k-1}(\Omega)$, see Remark \ref{analytic}. On the other hand it is expected (due to some numerical  computations) that the optimal set $\Omega$ for $\lambda_3$ in $\R^2$ is a ball and, in particular, $\lambda_3(\Omega)=\lambda_2(\Omega)$.
\end{oss}

\begin{oss}
The bound on the diameter of the optimal set $\Omega$, solution of \eqref{lbfperint}, depends on the function $f$ and on $\lambda_{k_p}(\Omega)$. In dimension two, this bound is trivially uniform, since it depends only on the perimeter constraint. This fact (together with he convexity of the minimizers) was used in \cite{bufr13} to study the asymptotic behaviour of the optimal sets $\Omega_k$, solutions of \eqref{lbko}. More precisely, it was proved that the sequence $\Omega_k\subset\R^2$ converges in the Hausdorff distance (up to translations) to the ball of unit perimeter. The analogous result in higher dimensions is not known yet.  
\end{oss}




\section{Preliminaries}\label{pr}
In this section we introduce the notions and results that we will need in the rest of the paper. As we saw in the introduction, we will have to solve partial differential equations on domains which are not open sets. For this purpose we extend the notion of a Sobolev space to any measurable set $\Omega\subset\R^d$, introducing the Sobolev-like spaces $\widetilde{H}^1_0(\Omega)$ defined as
\begin{equation}\label{sob}
\widetilde{H}^1_0(\Omega)=\left\{u\in H^1(\R^d):\ u=0\ \hbox{a.e. on}\ \Omega^c\right\},
\end{equation} 
where the term almost everywhere (shortly a.e.) refers to the Lebesgue measure $|\cdot|$ on $\R^d$. 
\begin{oss}
Note that even for open sets $\Omega$ the above definition differs from the classical one, in which $H^1_0(\Omega)$ is the closure of the smooth functions with compact support $C^\infty_c(\Omega)$ with respect to the norm
 \[
 \|\varphi\|^2_{H^1}:=\|\varphi\|_{L^2}^2+\|\nabla\varphi\|_{L^2}^2.
 \]
 To see that, one may take for example $\Omega$ to be the unit ball minus a hyperplane passing through the origin. Nevertheless, we have equality $H^1_0(\Omega)=\widetilde{H}^1_0(\Omega)$, if $\Omega$ is a bounded open set with Lipschitz boundary or, more generally, an open set satisfying a uniform exterior  density estimate (Proposition \ref{soblike}).  
\end{oss}

\begin{oss}\label{sobcap}
In the case of measure constraint another definition of a Sobolev space is used. Indeed, for any measurable set $\Omega\subset\R^d$ one may consider 
\begin{equation*}
H^1_0(\Omega)=\left\{u\in H^1(\R^d):\ \cp(\{u\neq0\}\cap\Omega^c)=0\right\},
\end{equation*} 
where the capacity $\cp(E)$ of a generic set $E\subset\R^d$ is defined as
\begin{equation}\label{sobcap1}
\cp(E)=\inf\left\{\|u\|_{H^1}:\ u\in H^1(\R^d),\ u=1\ \hbox{on a neighbourhood of}\ E\right\}.
\end{equation} 
In \cite[Theorem 3.3.42]{hepi05} it was proved that the above definition coincides with the classical one on the open sets of $\R^d$. There is a close relation between the Sobolev spaces as defined in \eqref{sobcap1} and Sobolev-like spaces from \eqref{sob}. It fact, for every measurable set $\Omega\subset\R^d$, we have the inclusion $H^1_0(\Omega)\subset\widetilde{H}^1_0(\Omega)$. Moreover, since $\widetilde H^1_0(\Omega)$ is separable, it is not hard to check that there is a measurable set $U\subset\R^d$ such that $U\subset\Omega$ a.e. and $H^1_0(U)=\widetilde{H}^1_0(\Omega)$ (take, for example, $U$ to be  the union of the supports of Sobolev functions, which form a dense subset of $\widetilde H^1_0(\Omega)$). The reason to work with definition \eqref{sob} instead of \eqref{sobcap1} is that we do not know the relation (if any) between the perimeter of $\Omega$ and the perimeter of $U$. 
\end{oss}

For any $\Omega\subset\R^d$ of finite measure and any $f\in L^2(\R^d)$, we define $R_\Omega(f)\in \widetilde H^1_0(\Omega)$ as the weak solution, in $\widetilde H^1_0(\Omega)$, of the equation
\begin{equation}\label{eqf}
-\Delta u=f,\qquad u\in\widetilde{H}^1_0(\Omega),
\end{equation}
or, equivalently, the unique minimizer in $\widetilde H^1_0(\Omega)$ of the convex functional
$$J_f(u)=\frac12\int_\Omega |\nabla u|^2\,dx-\int_\Omega fu\,dx.$$
Applying the Sobolev inequality in $\R^d$ and using $R_\Omega(f)$ as a test function in \eqref{eqf}, we have 
$$\|R_\Omega(f)\|^2_{L^2}\le |\Omega|^{2/d}\|R_\Omega(f)\|^2_{L^{\frac{2d}{d-2}}}\le C_d|\Omega|^{2/d}\|\nabla (R_\Omega(f))\|^2_{L^2}\le C_d|\Omega|^{2/d}\|R_\Omega(f)\|_{L^2}\|f\|_{L^2},$$
and so, $R_\Omega:L^2(\R^d)\to L^2(\R^d)$ is a bounded linear symmetric operator such that
$$\|R_\Omega\|_{\mathcal{L}(L^2(\R^d))}\le C_d|\Omega|^{2/d},\qquad R_\Omega(L^2(\R^d))\subset\widetilde{H}^1_0(\Omega).$$
Moreover, $R_\Omega$ is compact and positive and so, it has a discrete spectrum $\sigma(R_\Omega)$ contained in $\R^+$, which we write as 
$$0\le\dots\le\Lambda_k(\Omega)\le\dots\le\Lambda_1(\Omega).$$
We define $\widetilde\lambda_k(\Omega)$ as the inverse of $\Lambda_k(\Omega)$. Moreover, we have the variational characterization, analogous to the one in \eqref{lambdak}:
\begin{equation}
\widetilde\lambda_k(\Omega)=\min_{K\subset \widetilde{H}^1_0(\Omega)}\,\max_{u\in K}\,\,\frac{\int_{\Omega} |\nabla u|^2\,dx}{\int_{\Omega} u^2\,dx},
\end{equation}
where the minimum is over all $k$-dimensional linear subspaces $K$ of $\widetilde{H}^1_0(\Omega)$.

\bigskip
   For any measurable set $\Omega$ of finite Lebesgue measure, we denote with $w_\Omega$ the weak solution of the equation
\begin{equation}\label{eq1}
-\Delta w_\Omega=1,\qquad w_\Omega\in\widetilde{H}^1_0(\Omega),
\end{equation}   
i.e. the $w_\Omega$ is the unique minimizer in $\widetilde{H}^1_0(\Omega)$ of the functional 
\begin{equation}\label{eq2}
J(w)=\frac12\int_{\Omega}|\nabla w|^2\,dx-\int_\Omega w\,dx.
\end{equation}
We define the energy functional $\widetilde E(\Omega)$ as 
\begin{equation}\label{eq3}
\widetilde{E}(\Omega)=\min_{w\in \widetilde{H}^1_0(\Omega)}J(w)=J(w_\Omega)=-\frac12\int_\Omega w_\Omega\,dx,
\end{equation}
where to obtain the last equality we have used the definition of \(J\) and the relation
\[
\int_\Omega |\nabla w_\Omega|^2\,dx=\int_\Omega w_\Omega\,dx,
\]
which follows after testing \eqref{eq1} with \(w_\Omega\).
\begin{oss}\label{propenfun}
For the solution of \eqref{eq1} we have the estimates 
\begin{equation}\label{propenfun1}
\|w_\Omega\|_{L^2}\le C_d|\Omega|^{1+2/d},\qquad\|\nabla w_\Omega\|^2_{L^2}\le C_d|\Omega|^{3/4+2/d}.
\end{equation}
Moreover, by \cite{talenti},  $w_\Omega\in L^\infty(\R^d)$ and 
\begin{equation}\label{propenfun2}
\|w_\Omega\|_{L^\infty}\le C_d|\Omega|^{2/d},
\end{equation}
where $C_d$ is a constant depending only on the dimension. We also note that in the framework of Sobolev-like spaces, the weak maximum principle still holds ($w_U\le w_\Omega$, whenever $U\subset\Omega$), while the analogous of the strong maximum principle is the following equality (see Proposition \ref{soblike})
\begin{equation}\label{propenfun3}
H^1_0(\{w_\Omega>0\})=\widetilde{H}^1_0(\{w_\Omega>0\})=\widetilde{H}^1_0(\Omega),
\end{equation}
i.e. one may take $U=\{w_\Omega>0\}$ in Remark \ref{sobcap}. 
\end{oss} 

The relation between the operator $R_\Omega$ and the solution $w_\Omega$, is explained in the following two propositions
\begin{prop}\label{prop33}
Suppose that $\Omega_n\subset\R^d$ is a sequence of measurable sets of uniformly bounded Lebesgue measure and suppose that $w_{\Omega_n}$ converges to some $w\in H^1(\R^d)$ strongly in $L^2(\R^d)$, as $n\to\infty$. Then, there is a compact self-adjoint operator $R$ on $L^2(\R^d)$ such that $R_{\Omega_n}$ converges to $R$ in the strong operator topology in $\mathcal{L}(L^2(\R^d))$ and $R\le R_\Omega$ in sense of operators, i.e.
$$\int_{\R^d}R(f)f\,dx\le \int_{\R^d}R_\Omega(f)f\,dx,\ \hbox{for every}\ f\in L^2(\R^d),$$
where we set $\Omega=\{w>0\}$.  
\end{prop}
\begin{proof}
See \cite[Proposition 3.3]{buc00}.
\end{proof}

\begin{prop}\label{lemma36}
Suppose that $\Omega_1$ and $\Omega_2$ are two sets of finite measure in $\R^d$ such that $\Omega_1\subset\Omega_2$. Then, there are a positive constant $C$ and a real number $\theta\in(0,1)$, depending on the measure of $\Omega_2$ and the dimension $d$, such that 
$$\|R_{\Omega_1}-R_{\Omega_2}\|_{\mathcal{L}(L^2(\R^d))}\le C\|w_{\Omega_1}-w_{\Omega_2}\|^\theta
_{L^2(\R^d)}$$
\end{prop}
\begin{proof}
See \cite[Lemma 3.6]{buc00}.
\end{proof}




\section{Existence of generalized solutions}\label{ex}
In this section, we prove that for every $k\in\N$ and every $c>0$ the problem 
\begin{equation}\label{lbk}
\min\left\{\widetilde\lambda_k(\Omega):\ \Omega\subset\R^d,\ \Omega\ \hbox{measurable},\ P(\Omega)= c,\ |\Omega|<\infty\right\},
\end{equation} 
has a solution (see Theorem \ref{th1}). More generally, in Theorem \ref{thfper} we prove that there is a solution of 
\begin{equation}\label{lbfper}
\min\left\{f\left(\widetilde\lambda_{k_1}(\Omega),\dots,\widetilde\lambda_{k_p}(\Omega)\right):\ \Omega\subset\R^d,\ \Omega\ \hbox{measurable},\ P(\Omega)= c,\ |\Omega|<\infty\right\},
\end{equation} 
where $f$ is an increasing function satisfying \((f1), (f2), (f3)\). We will use a combination of a concentration-compactness principle and an induction argument, as in \cite{bulbk}. In order to deal with the \emph{dichotomy} case, we will show, on each step of the induction, that the optimal sets are  bounded. The following theorem is a straightforward adaptation of \cite[Theorem 2.2]{buc00} and already appeared in \cite{bubuhe}. We report the detailed proof in the Appendix for the sake of completeness.

\begin{teo}\label{cc}
Suppose that $\Omega_n\subset\R^d$ is a sequence of measurable sets of finite measure and uniformly bounded perimeter. Then, up to a subsequence, one of the following three situations occurs:
\begin{enumerate}[(i)]
\item\emph{Compactness:} There is a sequence $y_n\in\R^d$, a measurable set $\Omega\subset\R^d$ and a bounded self-adjoint operator $R:L^2(\R^d)\to L^2(\R^d)$, such that $\1_{y_n+\Omega_n}\to \1_{\Omega}$ in $L^1(\R^d)$, $R_{y_n+\Omega_n}\to R$ in $\mathcal{L}(L^2(\R^d))$, and $R\le R_\Omega$.
\item\emph{Dichotomy:} There are sequences of measurable sets $A_n\subset\R^d$ and $B_n\subset\R^d$ such that:
\begin{enumerate}[(a)]
\item $A_n\cup B_n\subset\Omega_n$;
\item $d(A_n,B_n)\to\infty$, as $n\to\infty$;
\item $\liminf_{n\to\infty}|A_n|>0$ and $\liminf_{n\to\infty}|B_n|>0$;
\item $\limsup_{n\to\infty}\left(P(A_n)+P(B_n)-P(\Omega_n)\right)\le0$.
\item $\lim_{n\to\infty}\|R_{A_n\cup B_n}-R_{\Omega_n}\|_{\mathcal{L}(L^2(\R^d))}=0$.
\end{enumerate} 
\item\emph{Vanishing:} For every $\eps>0$ and every  $R>0$ there exists $N\in\N$ such that 
$$\sup_{n\ge N}\sup_{x\in\R^d}|\Omega_n\cap B_R(x)|\le\eps.$$
Moreover,  $\|R_{\Omega_n}\|_{\mathcal{L}(L^2(\R^d))}\to 0$ as $n\to\infty$.
\end{enumerate}
\end{teo}

\begin{oss}\label{simple}Notice that if, in the previous theorem, we assume 

\begin{equation}\label{sesto}
\sup_{n\in \N} \widetilde \lambda_1(\Omega_n)<+\infty,
\end{equation}
 then the \emph{vanishing} cannot occur. Indeed, due to the equality
\[
\|R_{\Omega_n}\|_{\mathcal{L}(L^2(\R^d))}=\frac{1}{\widetilde \lambda_1(\Omega_n)},
\]
the sequence of resolvents does not converge to zero in norm.
\end{oss}

\begin{oss}\label{simple2}If the \emph{compactness} occurs for the sequence $\Omega_n\subset\R^d$, then we have 
$$P(\Omega)\le\liminf_{n\to\infty}P(\Omega_n),\qquad \widetilde\lambda_k(\Omega)\le\liminf_{n\to\infty}\widetilde\lambda_k(\Omega_n),$$
for every $k\in\N$.
If the \emph{dichotomy} occurs, then we have 
$$\liminf_{n\to\infty}P(A_n\cup B_n)\le\liminf_{n\to\infty}P(\Omega_n),\qquad \liminf_{n\to\infty}\widetilde\lambda_k(A_n\cup B_n)\le\liminf_{n\to\infty}\widetilde\lambda_k(\Omega_n),$$
for every $k\in\N$.

\end{oss}

%

As we already mentioned in the beginning of this section, the proof of existence of a minimizer of \eqref{lbfper} uses an induction argument which depends on some mild qualitative property of the minimizer (see Theorem \ref{th1} and Theorem \ref{thfper}). Thus, we will prove first a stronger result which states that any eventual solution of \eqref{lbfper} is a bounded set. We introduce the following notion which turns out to be a powerful tool in the analysis of optimal sets (see \cite{bulbk,bubuve} and \cite{bucve} for similar techniques)

\begin{deff}\label{sub}
Let $\Omega\subset\R^d$ be of finite perimeter and finite Lebesgue measure. We say that $\Omega$ is a {\bf local energy subsolution with respect to the perimeter} or simply {\bf energy subsolution} if there are constants $\varepsilon=\varepsilon (\Omega)>0$ and $k=k(\Omega)>0$ such that for each measurable set $U \subset\Omega$ with the property 
\begin{equation}\label{sub0}
\|w_\Omega-w_U\|_{L^2}\le\varepsilon,
\end{equation}
where $w_\Omega$ and $w_U$ are the solutions of \eqref{eq1}, we have 
\begin{equation}\label{sub1}
\widetilde E(\Omega)+kP(\Omega)\le \widetilde E(U)+k P(U).
\end{equation}
\end{deff}

\begin{prop}\label{lag}
Let $f$ be a function satisfying the assumptions $(f1)$, $(f2)$ and $(f3)$. Then any solution $\Omega\subset\R^d$ of \eqref{lbfper} is an energy subsolution.
\end{prop}
\begin{proof}
Let $U\subset\Omega$ and $t=\left(P(\Omega)/P(U)\right)^{1/(d-1)}$. Suppose that $t>1$, i.e. $P(U)<P(\Omega)$. By the optimality of $\Omega$, properties \((f2), (f3)\), the trivial scaling properties  of the eigenvalues and of the perimeter and the monotonicty of eigenvalues with respect to set inclusion, we obtain
\begin{equation*}
\begin{split}
0&\le f\left(\widetilde\lambda_{k_1}(tU),\dots,\widetilde\lambda_{k_p}(tU)\right)-f\left(\widetilde\lambda_{k_1}(\Omega),\dots,\widetilde\lambda_{k_p}(\Omega)\right)\\
&=f\left(\widetilde\lambda_{k_1}(tU),\dots,\widetilde\lambda_{k_p}(tU)\right)-f\left(\widetilde\lambda_{k_1}(U),\dots,\widetilde\lambda_{k_p}(U)\right)\\
&+f\left(\widetilde\lambda_{k_1}(U),\dots,\widetilde\lambda_{k_p}(U)\right)-f\left(\widetilde\lambda_{k_1}(\Omega),\dots,\widetilde\lambda_{k_p}(\Omega)\right)\\
&\le \frac{a}{p}\big(P(\Omega)\big)^{-\frac{2}{d-1}}\left(\sum_{j=1}^p\widetilde\lambda_{k_j}(U)\right)\left(P(U)^{\frac{2}{d-1}}-P(\Omega)^{\frac{2}{d-1}}\right)\\
&+L\sum_{j=1}^p\left(\widetilde\lambda_{k_j}(U)-
\widetilde\lambda_{k_j}(\Omega)\right)
\end{split}
\end{equation*}
where $L$ is the (local) Lipschitz constant of $f$ and $a$ is the constant from $(f3)$.  Using the concavity of the function \(z\mapsto z^{\frac 
{2} {d-1}}\) if \(d\ge 3\), or the fact that \(P(U)<P(\Omega)\) if \(d=2\), we can bound
\[
P(U)^{\frac{2}{d-1}}-P(\Omega)^{\frac{2}{d-1}}\le C(\Omega)\left(P(U)-P(\Omega)\right).
\]
By Proposition \ref{lemma36}, we have the estimate
\begin{equation}\label{lagper1}
0\le \widetilde\lambda_{k_j}(U)-\widetilde\lambda_{k_j}(\Omega)\le C\|w_\Omega-w_U\|^\theta_{L^2},\ \forall j=1,\dots,p.
\end{equation}
Thus, there are constants $\Lambda(\Omega)>0$ and $\eps(\Omega)>0$ such that for each $U\subset\Omega$ with the property 
$$P(U)<P(\Omega), \qquad \|w_\Omega-w_U\|_{L^2}\le\varepsilon,$$
we have 
\begin{equation}\label{maldipancia}
\sum_{j=1}^p\widetilde\lambda_{k_j}(\Omega)+\Lambda P(\Omega)\le \sum_{j=1}^p\widetilde\lambda_{k_j}(U)+\Lambda P(U).
\end{equation}
Moreover, thanks to the monotonicity of the eigenavlues by set inclusion, the above inequality trivially holds also if \(P(\Omega)\le P(U)\).

 We are now in the  position to prove that any solution $\Omega$ of \eqref{lbfper} is an energy subsolution. Indeed, by  \cite[Lemma 4.1]{bulbk}, for every  $U\subset\Omega$ we have
\[
 \widetilde\lambda_{k_j}(U)-\widetilde\lambda_{k_j}(\Omega)\le C \widetilde\lambda_{k_j}(U)\left(\widetilde E(U)-\widetilde E(\Omega)\right),
\]   
which together with    \eqref{maldipancia} (which holds for every \(U\subset \Omega\) satisfying \eqref{sub0}) and the monotonicity of the eigenavalues by set inclusion,  give
$$\Lambda\left(P(\Omega)-P(U)\right)\le \widetilde\lambda_{k_j}(U)-\widetilde\lambda_{k_j}(\Omega)\le C \widetilde\lambda_{k_j}(U)\left(\widetilde E(U)-\widetilde E(\Omega)\right),$$
where $C$ is a constant depending on $\Omega$ and $k_p$. Using again \eqref{lagper1}, we see that all the  $\widetilde\lambda_{k_j}(U)$, for $j=1,\dots,p$, remain bounded by a constant depending only on \(\Omega\) as $E(U)-E(\Omega)\le\eps$, concluding the proof.
\end{proof}

\begin{oss}
In the case $p=1$ and $f(x)=x$, the proof of Proposition \ref{bounded} can be simplified due to the following fact: \emph{There exists a positive constant $\Lambda>0$ such that any solution $\Omega\subset\R^d$ of \eqref{lbk} is also a solution of the problem}
\begin{equation}\label{lbklag}
\min\left\{\widetilde\lambda_k(\Omega)+\Lambda P(\Omega):\ \Omega\subset\R^d,\ \Omega\ \hbox{measurable},\ |\Omega|<+\infty\right\}.
\end{equation}
In order to prove that, we first notice that, by the scaling properties of $\lambda_k$  and of  the perimeter, $\Omega\subset\R^d$ is a solution of \eqref{lbklag}, if and only if, the following two conditions are satisfied:
\begin{enumerate}[(1)]
\item $\Omega$ is a solution of \eqref{lbk} with $c=P(\Omega)$.
\item The function $F(t)=\lambda_k(t\Omega)+\Lambda P(t\Omega)$, defined on the positive real numbers, achieves its minimum in $t=1$.  
\end{enumerate}
Thus, if $\Omega$ is a solution of \eqref{lbk}, then it is sufficient to choose $\Lambda>0$ such that the derivative 
$$F'(t)=-2\lambda_k(\Omega)t^{-3}+(d-1)\Lambda P(\Omega)t^{d-2},$$
vanishes in $t=1$, i.e.
\begin{equation}\label{lag1}
\Lambda=\frac{2\lambda_k(\Omega)}{(d-1)P(\Omega)}.
\end{equation}
\end{oss}

\begin{lemma}\label{bounded}
Let $\Omega$ be an energy subsolution. Then $\Omega$ is a bounded set.
\end{lemma}
\begin{proof}
For each $t\in\R$, we set 
\begin{equation}
H_t=\{x\in\R^d:\ x_1=t\},\qquad H_t^+=\{x\in\R^d:\ x_1>t\},\qquad H_t^-=\{x\in\R^d:\ x_1<t\}.
\end{equation}
We prove that there is some $t\in\R$ such that $|H_t^+\cap \Omega|=0$. For sake of simplicity, set $w:=w_\Omega$ and $M=\|w\|_{L^\infty}$. For any $t\in\R$ consider the function
\begin{equation}
v_t(x_1,\dots,x_d)=\begin{cases}
\begin{array}{ll}
M&,\ x_1\le t-\sqrt M,\\
\frac12\Big(2M-(x_1-t+\sqrt {2M})^2\Big)&, t-\sqrt {2M}\le x_1\le t,\\
0&,\ t\le x_1.
\end{array}
\end{cases}
\end{equation}

Consider the set $\Omega_t=\Omega\cap H_t^-$, obtained ``cutting'' \(\Omega\) with an hyperplane, and the function $w_t=w\wedge v_t\in \widetilde H^1_0(\Omega_t)$. We recall that $w_{\Omega_t}$ is the orthogonal projection of $w$ on $\widetilde{H}^1_0(\Omega_t)$ with respect to the $H_0^1(\Omega)$ scalar product\footnote{We recall that, thanks to the Poincar\`e inequality, \(\widetilde H^1_0(\Omega)\) is an Hilbert space with the scalar product given by \[\langle u,v\rangle_{\widetilde H_0^1(\Omega)} =\int \nabla u\cdot \nabla v\,dx. \]}, hence
$$\|w_{\Omega_t}-w\|_{L^2}\le C(\Omega)\|\nabla w-\nabla w_{\Omega_t}\|_{L^2}\le C(\Omega)\|\nabla(w-w_t)\|_{L^2},$$
for some constant $C(\Omega)$ depending on $\Omega$. Thus, for $t$ big enough, we have $E(\Omega_t)-E(\Omega)\le\eps$.  
Hence, we can use $\Omega_t$ as a competitor in \eqref{sub1}. We thus get the inequality:
$$J(w)+kP(\Omega)=\widetilde E(\Omega)+kP(\Omega)\le\widetilde E(\Omega_t)+kP(\Omega_t)\le J(w_t)+kP(\Omega_t),$$
where the functional $J$ is defined in \eqref{eq2}. Hence we get
$$\frac{1}{2}\int_{\Omega}|\nabla w|^2\,dx-\int_\Omega w\,dx+kP(\Omega)\le \frac{1}{2}\int_{\Omega_t}|\nabla w_t|^2\,dx-\int_{\Omega_t} w_t\,dx+kP(\Omega_t).$$
Notice that $w_t=0$ on $H_t^+$ and $w_t=w$ on $H_{t-\sqrt{2 M}}^-$. Setting $t_-=t-\sqrt{2M}$, and using the inequality
\begin{equation*}\label{disuvec}
\frac{|a|^2}{2}-\frac{|b|^2}{2}\le a\cdot (a-b)\qquad \forall\, a,b \in \R^d,
\end{equation*}
 we obtain
\begin{equation*}\label{bounded1}
\begin{split}
\frac{1}{2}\int_{H_t^+}|\nabla w|^2\,dx+k(P(\Omega)-P(\Omega_t))&\le \frac{1}{2}\int_{\{t_-<x_1<t\}}|\nabla w_t|^2-|\nabla w|^2\,dx+\int_{H_{t_-}^+}(w-w_t)\,dx\\
&\le \int_{\{t_-<x_1<t\}}\nabla w_t\cdot\nabla (w_t-w)\,dx+\int_{H_{t_-}^+}(w-w_t)\,dx\\
&=-\int_{\{t_-<x_1<t\}}\nabla v_t\cdot\nabla (w-v_t)_+\,dx+\int_{H_{t_-}^+}(w-v_t)_+\,dx\\
&=-\int_{H_t}w\,\dfrac{\partial v_t}{\partial x_1}\,d\H^{d-1}+\int_{H_t^+}w\,dx\\
&=\sqrt{2M}\int_{H_t}w\,d\H^{d-1}+\int_{H_t^+}w\,dx
\end{split}
\end{equation*}
where for a generic $u\in H^1(\R^d)$, with $u_+:=\sup\{u,0\}$ we indicate the positive part of $u$. Using again the boundedness of $w$, we get
\begin{equation}\label{bounded2}
k(P(\Omega, H^+_t)-P(H^+_t,\Omega))\le \sqrt 2 M^{3/2}\H^{d-1}(H_t\cap \Omega)+M|\Omega\cap H^+_t|.
\end{equation}
On the other hand, by the isoperimetric inequality, for almost every \(t\) we have 
\begin{equation}\label{bounded3}
|\Omega\cap H_t^+|^{\frac{d-1}{d}}\le C_d P(\Omega\cap H_t^+)= C_d\left(\H^{d-1}(H_t\cap \Omega)+P(\Omega, H_t^+)\right)\\
\end{equation}
Putting together \eqref{bounded2} and \eqref{bounded3} we obtain
\begin{equation}\label{bounded4}
|\Omega\cap H_t^+|^{\frac{d-1}{d}} \le C_1\left(\H^{d-1}(H^+_t\cap \Omega)+|\Omega\cap H_t^+|\right),
\end{equation}
where $C_1$ is some constant depending on the dimension $d$, the constant $k$ and the norm $M$. Setting $\phi(t)=|\Omega\cap H_t^+|$, we have that $\phi(t)\to0$ as $t\to+\infty$ and $\phi'(t)=-\H^{d-1}(H_t\cap\Omega)$. Chosing \(T=T(\Omega)\) such that
\[
C_1\phi(t)\le \frac 1  2 \phi(t)^{\frac{d-1}{d}} \qquad \forall \, t\ge T,
\]
equation \eqref{bounded4} gives
\begin{equation*}
\phi'(t)\le -2C_1\phi(t)^{1-1/d} \qquad \forall\, t\ge T,
\end{equation*}
which  implies that $\phi(\bar t)$ vanishes for some $\bar t\in\R$. Repeating this argument in any direction, we obtain that $\Omega$ is bounded.
\end{proof}

We are now in position to prove the existence of an optimal set for \eqref{lbfper}. We first prove the result for the problem \eqref{lbk}. This is just a particular case of \eqref{lbfper}, but it will be the first step of the proof of the more general Theorem \ref{thfper}, which is based on the same idea.

\begin{teo}\label{th1}
For any $k\in\N$, there exists a solution $\Omega\subset\R^d$ of \eqref{lbk}. Moreover, any solution of \eqref{lbk} is a bounded set.
\end{teo}
\begin{proof}
Without loss of generality we assume that $c=1$. We prove the theorem by induction on $k\in\N$. For $k=1$ the existence holds, since by a standard symmetrization argument, we have that the optimal set is a ball of perimeter $1$.

Let $k>1$ and let $\Omega_n$ be a minimizing sequence for \eqref{lbk}. We note that clearly the perimeters are bounded and the sets have finite measures, hence the assumptions of Theorem \ref{cc} are satisfied. Moreover by the monotonicity of the eigenvalues
\[
\limsup_{n \to \infty} \widetilde \lambda_1(\Omega_n)\le \limsup_{n \to \infty} \widetilde \lambda_k(\Omega_n)<+\infty,
\]
hence by Remark \ref{simple} we have only   two possibilities:
\begin{enumerate}[(i)]
\item\emph{Compactness:} Since $y_n+\Omega_n$ is also a minimizing sequence, we have that the limit $\Omega$ is such that 
\[
P(\Omega)\le 1 \quad\text{ and }\quad \widetilde \lambda_k(\Omega)\le\liminf_{n\to\infty}\widetilde \lambda_k(\Omega_n).
\]
Thus $\Omega$ is a solution of \eqref{lbk}.
\item\emph{Dichotomy:} By the dichotomy case of Theorem \ref{cc} and a scaling argument, we can construct a sequence of sets  
\[
\widetilde\Omega_n=A_n\cup B_n,
\]
satisfying \(P(\widetilde \Omega_n)=1\), $d(A_n,B_n)\to\infty$ and which is still a minimizing sequence
 for \eqref{lbk}. 
 Without loss of generality and up to extracting a subsequence\footnote{We recall that if \(\Omega=A\cup B\) with \(\dist(A,B)>0\), \(\widetilde H_0^1(\Omega)=\widetilde H^1_0(A)\oplus \widetilde H_0^1(B)\) hence the spectrum of the Dirichlet Laplacian of \(\Omega\) is given by the union of the spectrum of the Dirichlet Laplacian of \(A\) and of \(B\).}, if necessary, we may assume that $\lambda_k(A_n\cup B_n)=\lambda_l(A_n)$ for some $l\le k$. We note that since $\widetilde\Omega_n$ is minimizing, we must have $l<k$. Indeed, if this is not the case, i.e. $l=k$, the sequence $A_n$ is such that $\widetilde\lambda_k(A_n)=\widetilde\lambda_k(\widetilde\Omega_n)$ but on the other hand, by Theorem \ref{cc} 
$$\limsup_{n\to\infty}P(A_n)\le 1-\liminf_{n\to\infty}P(B_n)<1,$$
where the strict inequality is due to point $(c)$ of the dichotomy case in Theorem \ref{cc} and the isoperimetric inequality
\[
\liminf_{n\to \infty} P(B_n)\ge C_d \liminf_{n\to \infty} |B_n|^{\frac{d-1}{d}}>0.
\]
Thus an appropriate rescaling of $A_n$ would lead to a strictly better minimizing sequence, which is impossible. Hence \(l<k\), taking the biggest possible $l$, we can assume $\widetilde\lambda_l(A_n)<\widetilde\lambda_{l+1}(A_n)$. Since the  spectrum of the Dirichlet Laplacian of \(\widetilde\Omega_n\) is given by the union of the spectrum of the the Laplacian on \(A_n\) and \(B_n\), this assumption implies 
\begin{equation}\label{corridoio}
\max\big\{\widetilde\lambda_l(A_n),\widetilde\lambda_{k-l} (B_n)\big\}=\widetilde\lambda_l(A_n)=\widetilde\lambda_k(\widetilde\Omega_n).
\end{equation}

Up to a extract a subsequence, we suppose that the following limits exist:
$$\lim_{n\to\infty}P(A_n),\qquad \lim_{n\to\infty}P(B_n),\qquad \lim_{n\to\infty}\widetilde\lambda_l(A_n),\qquad \lim_{n\to\infty}\widetilde\lambda_{k-l}(B_n).$$
Let $\Omega_l^\ast$ and $\Omega_{k-l}^\ast$ be the optimal sets for \eqref{lbk} for $\lambda_l$ with constraint $c_A=\lim_{n\to\infty}P(A_n)$ and $\lambda_{k-l}$ with constraint $c_B=\lim_{n\to\infty}P(B_n)$, respectively. Since, thanks to Lemma \ref{bounded}, they are bounded, we may suppose that up to translation, they are at positive distance. Then
\[
P(\Omega_l^\ast\cup\Omega_{k-l}^\ast)=P(\Omega_l^\ast)+P(\Omega_{k-l}^\ast)=\lim_{n\to \infty} P(A_n)+\lim_{n\to \infty} P(B_n) \le 1
\]
and, by the choice of \(\Omega_l^\ast\) and \(\Omega_{k-l}^\ast\),
\[
\widetilde \lambda_l(\Omega_l^\ast)\le \lim_{n\to\infty}\widetilde\lambda_l(A_n),\qquad\widetilde \lambda_{k-l}(\Omega_{k-l}^\ast)\le \lim_{n\to\infty}\widetilde \lambda_{k-l}(B_n).
 \]
As a consequence, thanks to \eqref{corridoio},
\[
\begin{split}
\widetilde \lambda_k(\Omega_l^\ast\cup\Omega_{k-l}^\ast)&=\max\big\{\widetilde \lambda_l(\Omega_l^\ast),\widetilde \lambda_{k-l}(\Omega_{k-l}^\ast)\big\}\\
&\le \lim_{n\to\infty}\max\big\{\widetilde\lambda_l(A_n),\widetilde \lambda_{k-l}(B_n)\big\}\\
& =\liminf _{n\to\infty}\widetilde \lambda_{k}(\widetilde \Omega_n).
\end{split}
\]
Hence $\Omega_l^\ast\cup\Omega_{k-l}^\ast$ is a solution of \eqref{lbk}.
\end{enumerate}
\end{proof}

We now prove the existence in the general case of a spectral functional of the form $F(\Omega)=f\left(\widetilde\lambda_1(\Omega),\dots,\widetilde\lambda_k(\Omega)\right)$. 

\begin{teo}\label{thfper}
Suppose that $f$ satisfies $(f1)$, $(f2)$ and $(f3)$. Then there exists a solution $\Omega\subset\R^d$ of \eqref{lbfper}.
Moreover, every solution of \eqref{lbfper} is a bounded set.
\end{teo}
\begin{proof}
As in the proof of Theorem \ref{th1}, we proceed by induction, this time on the number of variables $p$. If $p=1$, then thanks to the monotonicity of \(f\), any solution of \eqref{lbk} is also a solution of \eqref{lbfper} and so we have the claim by Theorem \ref{th1}. 

Consider now the functional $F(\Omega)=f(\widetilde\lambda_{k_1}(\Omega),\dots,\widetilde\lambda_{k_p}(\Omega))$ and let $\Omega_n$ be a minimizing sequence. By Theorem \ref{cc} and using \((f1)\) and Remark \ref{simple}, there are two possible behaviours for the sequence $\Omega_n$: compactness and dichotomy. 

\noindent
If the compactness occurs, we  immediately obtain the existence of an optimal set. Otherwise, in the dichotomy case,  we may suppose that $\Omega_n=A_n\cup B_n$, where the Lebesgue measure of $A_n$ and $B_n$ is uniformly bounded from below and $\dist(A_n,B_n)\to\infty$. Moreover, up to a extracting a subsequence, we may suppose that there is some $1\le l<p$ and two sets of natural numbers
$$1\le\alpha_1<\dots<\alpha_l,\qquad 1\le\beta_{l+1}<\dots<\beta_{p},$$
such that for every $n\in\N$, we have
$$\left\{\widetilde\lambda_{\alpha_1}(A_n),\dots,\widetilde\lambda_{\alpha_l}(A_n),\widetilde\lambda_{\beta_{l+1}}(B_n),\dots,\widetilde\lambda_{\beta_{p}}(B_n)\right\}=\left\{\widetilde\lambda_{k_1}(\Omega_n),\dots,\widetilde\lambda_{k_p}(\Omega_n)\right\}.$$
Indeed, if all the eigenvalues of \(\Omega_n\) are realized by, say, \(A_n\) arguing as in the proof of Theorem \ref{th1} we can construct a strictly better minimizing sequence.
Moreover, without loss of generality we may suppose that
$$\widetilde\lambda_{\alpha_i}(A_n)=\widetilde\lambda_{k_i}(\Omega_n), \forall i=1,\dots,l;\qquad \widetilde\lambda_{\beta_j}(B_n)=\widetilde\lambda_{k_j}(\Omega_n), \forall j=l+1,\dots,p.$$ 
We can also assume that the sequences $\widetilde\lambda_{\alpha_i}(A_n)$ and $\widetilde\lambda_{\beta_j}(B_n)$ converge as $n\to\infty$. By a scaling argument as in Theorem \ref{th1} we also have that without loss of generality $P(A_n)=c_\alpha$ and $P(B_n)=c_\beta$, where $c_\alpha$ and $c_\beta$ are fixed positive constants. Let $f_\alpha:\R^l\to\R$ be the restriction of $f$ to the $l$-dimensional hyperplane 
$$\left\{x\in\R^p:\ x_j=\lim_{n\to\infty}\widetilde\lambda_{\beta_j}(B_n),\ j=l+1,\dots,p\right\}.$$
Since $l<p$, by the inductive assumption, there is a solution $A^\ast$ of the problem 
\begin{equation}\label{lbfper1}
\min\left\{f_\alpha\left(\widetilde\lambda_{\alpha_1}(A),\dots,\widetilde\lambda_{\alpha_l}(A)\right):\ A\subset\R^d,\ A\ \hbox{measurable},\ P(A)= c_\alpha,\ |A|<\infty\right\}.
\end{equation}     
Since $f$ is locally Lipschitz, we have
\begin{equation*}
\begin{array}{ll}
\liminf_{n\to\infty}f\left(\widetilde\lambda_{\alpha_1}(A_n),\dots,\widetilde\lambda_{\alpha_l}(A_n),\widetilde\lambda_{\beta_{l+1}}(B_n),\dots,\widetilde\lambda_{\beta_p}(B_n)\right)\\
\\
\qquad=\liminf_{n\to\infty}f\left(\widetilde\lambda_{\alpha_1}(A_n),\dots,\widetilde\lambda_{\alpha_l}(A_n),\lim_{m\to\infty}\widetilde\lambda_{\beta_{l+1}}(B_m),\dots,\lim_{m\to\infty}\widetilde\lambda_{\beta_p}(B_m)\right),
\end{array}
\end{equation*}
and thus the minimum in \eqref{lbfper1} is smaller than the infimum in \eqref{lbfper}. Moreover, $A^\ast$ is bounded and so, we may suppose that $\dist(A^\ast,B_n)>0$, for all $n\in\N$. Thus, again by the Lipschitz condition on $f$, the sequence $A^\ast\cup B_n$ is minimizing for \eqref{lbfper}.

  Let now $f_\beta:\R^{p-l}\to\R$ be the restriction of $f$ to the $(p-l)$-dimensional hyperplane 
$$\left\{x\in\R^p:\ x_i=\widetilde\lambda_{\alpha_i}(A^\ast),\ i=1,\dots,l\right\}.$$
Let $B^\ast$ be a solution of the problem  
\begin{equation}\label{lbfper2}
\min\left\{f_\beta\left(\widetilde\lambda_{\beta_{l+1}}(B),\dots,\widetilde\lambda_{\beta_p}(B)\right):\ B\subset\R^d,\ B\ \hbox{measurable},\ P(B)= c_\beta,\ |B|<\infty\right\}.
\end{equation} 
We have that the minimum in \eqref{lbfper2} is smaller than the minimum in \eqref{lbfper1} and so than that in \eqref{lbfper}. On the other hand, since both $A^\ast$ and $B^\ast$ are bounded and the functionals we consider are translation invariant, we may suppose that $\dist(A^\ast,B^\ast)>0$. Thus the set $\Omega^\ast:=A^\ast\cup B^\ast$  is a solution of \eqref{lbfper}. 
\end{proof}




\section{Existence of an open solution }\label{exo}
In this section we study the shape optimization problems \eqref{lbko} and \eqref{lbfperint}. In Theorem \ref{thfpero} we prove that there exist solutions of \eqref{lbko} and \eqref{lbfperint} and we study their relation with the corresponding solutions of \eqref{lbk} and \eqref{lbfper}. Before we prove the theorem we need some preliminary results concerning the sets which, in some generalized sense, have positive mean curvature. 

\begin{deff}
We say that the measurable set $\Omega$ is a {\bf perimeter supersolution} if it has finite Lebesgue measure, finite perimeter and satisfies the following condition:
\begin{equation}\label{ext1}
P(\Omega)\le P(\widetilde\Omega),\ \hbox{for each}\ \ \widetilde\Omega\supset\Omega.
\end{equation} 
\end{deff}

\begin{oss}
Let $\Omega$ be an open set with boundary $\partial\Omega$ of class $C^2$. If $\Omega$ is a perimeter supersolution, then it has positive mean curvature with respect to the interior normal vector field on $\partial\Omega$. Lemma \ref{dist} below shows that, even if it is less regular, it has positive mean curvature in the viscosity sense.
\end{oss}

The following simple Lemma will play a crucial role in the sequel:
\begin{lemma}\label{super}
Suppose that $f:\R^p\to\R$ satisfies conditions $(f1)$, $(f2)$ and $(f3)$ and that $\Omega\subset\R^d$ is a solution of \eqref{lbfper}. Then $\Omega$ is a perimeter supersolution. 
\end{lemma}
\begin{proof}
Suppose, by contradiction, that $\widetilde\Omega\supset\Omega$ is such that $P(\widetilde\Omega)<P(\Omega)$ and set $$t=\left(P(\Omega)/P(\widetilde\Omega)\right)^{1/(d-1)}>1.$$
 Then, for any $k\in\N$, we have
\[
\widetilde \lambda_k (t\widetilde \Omega)<\widetilde \lambda_k (\widetilde \Omega)\le \widetilde \lambda_k(\Omega).
\]
On the other hand $P(t\widetilde\Omega)=P(\Omega)$ and so, by the optimality of $\Omega$ and the strict monotonicity of \(f\), \((f3)\), we have
\begin{equation*}
\begin{array}{ll}
0&\le f\left(\widetilde\lambda_{k_1}(t\widetilde\Omega),\dots,\widetilde\lambda_{k_p}(t\widetilde\Omega)\right)-f\left(\widetilde\lambda_{k_1}(\Omega),\dots,\widetilde\lambda_{k_p}(\Omega)\right)\\
\\
&<f\left(\widetilde\lambda_{k_1}(\widetilde\Omega),\dots,\widetilde\lambda_{k_p}(\widetilde\Omega)\right)-f\left(\widetilde\lambda_{k_1}(\Omega),\dots,\widetilde\lambda_{k_p}(\Omega)\right)\le 0,
\end{array}
\end{equation*}
which is a contradiction.
\end{proof}

The following result is classical (see, for instance, \cite{giusti}, \cite[Theorem 16.14]{maggi}) and so we only sketch the proof.
\begin{lemma}\label{ext}
Let $\Omega\subset\R^d$ be a perimeter supersolution. Then there exists a positive constant \(\bar c\), depending only on the dimension $d$, such that for every \(x\in \R^d\), one of the following situations occur:
\begin{enumerate}[(a)]
\item there is some ball $B_r(x)$ with $r>0$ such that $B_r(x)\subset \Omega$ a.e.,
\item for each ball $B_r(x)\subset\R^d$,  we have $|B_r(x)\cap\Omega^c|\ge \bar c|B_r|$.
\end{enumerate}
\end{lemma}
\begin{proof}
Let $x\in\R^d$. Suppose that there is no $r>0$ such that $B_r(x)\subset\Omega$. We will prove that $(b)$ holds. Using the condition \eqref{ext1} for $\widetilde\Omega=\Omega\cup B_r(x)$ we get that for almost every \(r\),
\[
P(\Omega, B_r(x))\le \H^{d-1}(\partial B_r(x)\cap \Omega^c).
\]
Applying the isoperimetric inequality to $B_r(x)\setminus\Omega$, we obtain
\begin{equation}\label{ext2}
\begin{split}
|B_r(x)\setminus\Omega|^{1-1/d}&\le C_d\left(P(\Omega,B_r(x))+\H^{d-1}(\partial B_r(x)\cap \Omega^c)\right)\\
\\
&\le 2C_d\H^{d-1}(\partial B_r(x)\cap \Omega^c).
\end{split}
\end{equation}
Consider the function $\phi(r)=|B_r(x)\setminus\Omega|$. Note that $\phi(0)=0$ and $\phi'(r)=\H^{d-1}(\partial B_r(x)\cap\Omega)$ and so, by \eqref{ext2}, 
 \[
 \bar c \le \frac{d}{dr}\left(\phi(r)^{1/d}\right),
 \]
which after integration gives (b).
\end{proof}

\begin{deff}
If $\Omega\subset\R^d$ is a set if finite Lebesgue measure and if there is a constant $\bar{c}>0$ such that for each point $x\in\R^d$ one of the conditions $(a)$ and $(b)$, from Lemma \ref{ext}, holds, then we say that $\Omega$ satisfies an {\bf exterior density estimate}. 
\end{deff}

\begin{prop}\label{est}
Let $\Omega\subset\R^d$ be a set of finite Lebesgue measure satisfying an exterior density estimate. Then there are positive constants $C$ and $\beta$ such that, for each $x\in\R^d$ with the property that $|B_r(x)\cap\Omega^c|>0$, for every $r\ge 0$, we have 
\begin{equation}\label{est2}
\|w_\Omega\|_{L^\infty(B_r(x))}\le r^\beta \|w_\Omega\|_{L^\infty(R^d)},\ \hbox{for each}\ r>0.
\end{equation}
In particular, if $\Omega$ is a perimeter supersolution, then the above conclusion holds.
\end{prop}
\begin{proof}
Let $x\in\R^d$ be such that that $|B_r(x)\cap\Omega^c|>0$, for every $r>0$. Without loss of generality we can suppose that $x=0$. Setting $w:=w_\Omega$, we have that $\Delta w+1\ge0$ in distributional sense on $\R^d$. Thus, on each ball $B_r(y)$ the function 
$$u(x)=w(x)-\frac{1}{2d}(r^2-|x-y|^2),$$
is subharmonic. By the mean value property
\begin{equation}\label{est3}
w(y)\le \frac{r^2}{2d}+\frac{1}{\omega_d\,r^d}\int_{B_r(y)}w(x)\,dx.
\end{equation}
Let us define \(r_n=4^{-n}\). For any $y\in B_{r_{n+1}}(0)$, equation \eqref{est3} implies
\begin{equation}\label{1136}
\begin{split}
w(y)&\le \frac {r^2_{n}}{4d}+\frac{1}{|B_{2r_{n+1}}(y)|}\int_{B_{2r_{n+1}}(y)} w(x)\,dx\\
\\
&\le  \frac {r^2_{n}}{4d}+\frac{|\Omega\cap B_{2r_{n+1}}(y)|}{|B_{2r_{n+1}}(y)|}\|w\|_{L^\infty(B_{2r_{n+1}(y)})}\\
\\
&\le  \frac {r^2_{n}}{4d}+\left(1-\frac{|\Omega^c\cap B_{r_{n+1}}(0)|}{|B_{2r_{n+1}}|}\right)\|w\|_{L^\infty(B_{r_{n}}(0))}\\
\\
&\le  \frac {4^{-2n}}{4d}+\left(1-2^{-d}\bar{c}\right)\|w\|_{L^\infty(B_{r_{n}}(0))},
\end{split}
\end{equation}
where in the third inequality we have used the inclusion \(B_{r_{n+1}}(0)\subset B_{2r_{n+1}}(y)\) for every \(y\in B_{r_{n+1}}(0)\). Hence setting
\[
a_n=\|w\|_{L^\infty(B_{r_{n}}(0))},
\]
we have
\[
a_{n+1}\le  \frac{8^{-n}}{4d} +(1- 2^{-d}\bar c)a_n,
\]
which easily implies \(a_{n}\le C a_0 4^{-n\beta}\) for some constants \(\beta\) and \(C\) depending only on \(\bar c\). This gives~\eqref{est2}. 
\end{proof}

\begin{prop}\label{soblike}
Let $\Omega\subset\R^d$ be a set of finite Lebesgue measure satisfying an external density estimate. Then the set 
$$\Omega_1:=\left\{x\in\R^d:\ \exists\lim_{r\to0}\frac{|\Omega\cap B_r(x)|}{|B_r(x)|}=1\right\},$$
is open and  $\widetilde{H}^1_0(\Omega)=H^1_0(\Omega_1)$. In particular, if $\Omega$ is a perimeter supersolution, then $\Omega_1$ is open and $\widetilde{H}^1_0(\Omega)=H^1_0(\Omega_1)$. 
\end{prop}
\begin{proof}
Thanks to Lemma \ref{ext},  $\Omega_1$ is an open set.  It remains to prove the equality between the Sobolev spaces. We first note that we have the equality $$\widetilde{H}^1_0(\Omega)= H^1_0(\{w_\Omega>0\}),$$
where $w_\Omega$ is as in Section \ref{pr}. Indeed, by the definitions of $H^1_0$ and $\widetilde H^1_0$, we have  
\[
H^1_0(\{w_\Omega>0\})\subset \widetilde{H}^1_0(\{w_\Omega>0\})\subset \widetilde{H}^1_0(\Omega),
\]
and so it is sufficient to prove that the inclusion $\widetilde{H}^1_0(\Omega)\subset H^1_0(\{w_\Omega>0\})$ holds. For, it is enough to check that for every positive and bounded $u\in\widetilde H^1_0(\Omega)$ we have $u\in H^1_0(\{w_\Omega>0\})$. Reasoning as in \cite[Proposition 3.1]{dmgar}, for every $u\in H^1_0(\Omega)$ such that $0\le u\le 1$ and every $n\in\N$, we consider the weak solution $u_n\in \widetilde H^1_0(\Omega)$ of the equation 
$$-\Delta u_n+nu_n=nu.$$
By the weak maximum principle, we have that $u_n\le n w_\Omega$ a.e. and so, by \cite[Lemma 3.3.30]{hepi05}, $u_n\le nw_\Omega$ quasi-everywhere\footnote{A property \(\mathcal P\) is said to hold quasi-everywhere if \[ {\rm cap}(\{\mathcal P \text{ is false}\})=0\].}, which implies that $u_n\in H^1_0(\{w_\Omega>0\})$. On the other hand, using $u_n-u\in \widetilde H^1_0(\Omega)$ as a test function, we have 
$$\int_\Omega \nabla u_n\cdot\nabla (u_n-u)\,dx+n\int_\Omega|u_n-u|^2\,dx=0,$$ 
and so 
\begin{equation}\label{soblikee1}
\int_\Omega|\nabla (u_n-u)|^2\,dx+n\int_\Omega|u_n-u|^2\,dx=-\int_\Omega\nabla u\cdot\nabla(u_n-u)\,dx\le\|\nabla u\|_{L^2} \|\nabla (u_n-u)\|_{L^2}.
\end{equation}
By \eqref{soblikee1}, we have that $\|\nabla (u_n-u)\|_{L^2}\le \|\nabla u\|_{L^2}$ and $\|u_n-u\|\le n^{-1}\|\nabla u\|_{L^2}$. Thus $u_n$ converges to $u$ weakly in $\widetilde H^1_0(\Omega)$ and, by the first equality in \eqref{soblikee1}, the convergence is also strong in $H^1(\R^d)$. As a consequence $u_n$ converges pointwise to $u$ outside a set of zero capacity (see \cite[Proposition 3.3.33]{hepi05}) and so we have $u\in H^1_0(\{w_\Omega>0\})$.
   
   We now prove that $\Omega_1=\{w_\Omega>0\}$ up to a set of zero capacity. Consider a ball $B\subset\Omega_1$. By the weak maximum principle, $w_B\le w_\Omega$ and so  
\[
\Omega_1\subset\{w_\Omega>0\}.
\]
In order to prove the other inclusion, we first note that since $w_\Omega\in H^1(\R^d)$, then there is a set $N\subset\R^d$ of zero capacity such that for any $x_0\in\R^d\setminus N$ there exists the limit 
$$\lim_{r\to 0}\frac{1}{|B_r(x_0)|}\int_{B_r(x_0)}w_\Omega\,dx=:\widetilde{w}_\Omega(x_0),$$
and $\widetilde{w}_\Omega$ coincides with $w_\Omega$ again up to a set of zero capacity (see \cite[Section 4.8]{evgar}). By Proposition \ref{est}, $\widetilde{w}_\Omega=0$ on $\R^d\setminus\Omega_1$ which gives the other inclusion.
\end{proof}

%

\begin{teo}\label{thfpero}
Let $k_1,\dots,k_p\in\N$ and suppose that $f:\R^p\to\R$ satisfies the assumptions $(f1)$, $(f2)$, and $(f3)$. Then there exists a solution $\Omega\subset\R^d$ of \eqref{lbfperint}.
Moreover, we have that
\begin{enumerate}[(i)]
\item every solution of \eqref{lbfperint} is a solution of \eqref{lbfper};
\item every solution of \eqref{lbfper} is equivalent to a solution of \eqref{lbfperint}.
\end{enumerate}
\end{teo}
\begin{proof}
Let $\Omega$ be a solution of \eqref{lbfper} and let $\Omega_1$ be as in Proposition \ref{soblike}. Let $ U\subset\R^d$ be an open set with perimeter $P(U)=1$. Using the inclusion $H^1_0(U)\subset \widetilde{H}^1_0(U)$, the optimality of $\Omega$ and Proposition \ref{soblike}, we have that for any $k\in\N$,
\begin{equation}\label{th2e1}
\lambda_k(U)\ge\widetilde\lambda_k(U)\ge \widetilde\lambda_k(\Omega)=\lambda_k(\Omega_1),
\end{equation}
i.e. $\Omega_1$ is a solution of \eqref{lbfperint}, which proves the existence part and claim $(ii)$. Suppose now that $U$ is an optimal set for \eqref{lbfperint}. Then all the equalities in \eqref{th2e1} must be equalities and so, we have claim $(i)$.
\end{proof}




\section{Regularity of the free boundary}\label{reg}
In this section we study the regularity of the reduced boundary $\partial\Omega$ of any solution $\Omega$ of \eqref{lbfperint}. The goal is to show that solutions of \eqref{lbfperint} are quasi-minimizers for the perimeter and then to apply some classical regularity results. In order to prove the quasi-minimality 
\begin{equation*}
P(\Omega,B_r)\le P(\widetilde\Omega,B_r)+Cr^{d},\qquad \forall\  \widetilde \Omega\Delta \Omega \subset B_r(x),
\end{equation*}
of the optimal set $\Omega$, we first note that, thanks to Propositions \ref{est}   and \ref{soblike}, \(w_\Omega\) is continuous. Moreover since \(\Omega\) is a perimeter supersolution we can show that \({\rm dist}(x, \Omega^c)\) is super harmonic in \(\Omega\) in the viscosity sense. Hence we can apply the maximum principle in order to prove that \(w_\Omega\) is Lipschitz. An estimate on the perimeter of the variations of an energy subsolution will finally give the desired quasi-minimality property of \(\Omega\).
%
%
   The following lemma is classical and so, we only give a reference for the proof.
\begin{lemma}\label{gradest}
Let $u\in H^1(B_r)$ be such that $-\Delta u=f\in L^\infty(B_r)$ in a weak sense in the ball $B_r$. Then we have 
\begin{equation}
\|\nabla u\|_{L^\infty(B_{r/2})}\le C_d\|f\|_{L^\infty(B_{r})}+\frac{2d}{r}\|u\|_{L^{\infty}(B_r)}.
\end{equation}
\end{lemma}
\begin{proof}
See \cite{GT}.
\end{proof}

\begin{prop}\label{holder}
Let $\Omega\subset\R^d$ be a perimeter supersolution. Then $w_{\Omega}:\R^d\to\R$ is H\"older continuous and
\begin{equation}\label{holder2}
\left|w_{\Omega}(x)-w_\Omega(y)\right|\le C|x-y|^{\beta},
\end{equation}
where $\beta$ is the constant from Proposition \ref{est}.
\end{prop}
\begin{proof}
Thanks to Proposition \eqref{soblike}, up to a set of capacity zero, we can assume that \(\Omega_1\) is open and that \(w_\Omega\) is the classical solution, with Dirichlet boundary conditions, of \(-\Delta w_{\Omega}=1\) in \(\Omega_1\).
Consider two distinct points $x,y\in\R^d$. In case both \(x\) and \(y\) belong to \(\Omega_1^c\), the estimate \eqref{holder2} is trivial. Let us assume that \(x\in \Omega_1\) and let \(x_0\in \partial \Omega_1\) be such that
\[
|x-x_0|=\dist(x,\partial \Omega_1).
\]
We distinguish two cases:
\begin{itemize}
\item Suppose that $y\in\R^d$ is such that
\[
2|x-y|\ge \dist(x,\partial \Omega_1).
\]
Hence \(x,y\in B_{4|x-y|}(x_0)\) and 
  by Proposition \ref{est}, we have that 
  \[
  w_{\Omega}(x)\le C|x-y|^\beta\quad \text{and}\quad  w_{\Omega}(y)\le C|x-y|^\beta.
  \]
Thus we obtain 
\begin{equation}\label{holder3}
\left| w_{\Omega}(x)-w_{\Omega}(y)\right|\le 2C|x-y|^{\beta}.
\end{equation} 
\item Assume that $y\in\R^d$ is such that
\[
2|x-y|\le \dist(x,\partial \Omega_1).
\]
Applying Lemma \ref{gradest} to \(w_{\Omega}\) in \(B_{\dist(x,\partial \Omega_1)}(x)\subset \Omega_1\) we obtain
\begin{equation}\label{holder4}
\|\nabla w_\Omega\|_{L^\infty(B_{\dist(x,\partial \Omega_1)/2}(x))}\le \frac{C_d\|w\|_{L^\infty(B_{\dist(x,\partial \Omega_1)}(x))}}{\dist(x,\partial \Omega_1)}\le C_d\dist(x,\partial \Omega_1)^{\beta-1},
\end{equation}
which, since \(\beta< 1\), together with our assumption and the mean value formula implies
\[
|w_{\Omega}(x)-w_{\Omega}(y)|\le C_d\dist(x,\partial \Omega_1)^{\beta-1}|x-y|\le |x-y|^{\beta}.
\]
\end{itemize} 
\end{proof}

In the following Lemma we show that a perimeter supersolution has positive mean curvature in the viscosity sense. This is done showing that the function \(d(x,\Omega^c)\) is super harmonic in \(\Omega\) in the viscosity sense  
(see \cite{CC} for a nice account of theory of viscosity solutions). In case \(\partial \Omega\) is smooth this easily implies that the mean curvature of \(\partial \Omega\), computed with respect to the interior normal, is positive (see for instance \cite[Section 14.6]{GT}). A similar observation  already appeared in \cite{CCo}, in the study of the regularity of minimal surfaces, and in \cite{jls, Ma}, in the study of free boundary type problems.
 
 We say that \emph{ \(\varphi\) touches \(d_\Omega\)} from below at \(x_0\) if
  \[
 d_\Omega(x_0)-\varphi(x_0)=\min_{\overline \Omega} \big\{d_\Omega-\varphi \big\}.
 \]
\begin{lemma}\label{dist}
Let $\Omega\subset\R^\dim$ be a perimeter supersolution. Consider the function $d_{\Omega}(x)=d(x,\Omega^c)$. Then for each $\varphi \in C^\infty_c(\Omega)$, touching \(d_{\Omega}\) from below at \(x_0\in \Omega\),   we have $\Delta \varphi (x_0)\le 0$.
\end{lemma}
\begin{proof}
Suppose, by contradiction, that there are  point $x_0\in\Omega$ and a function $\varphi \in C^\infty_c(\Omega)$ touching \(d_\Omega\) from below at \(x_0\) for which $\Delta \varphi (x_0)>0$. Up to a vertical translation, we can assume that \(\varphi (x_0)=d_\Omega(x_0)>0\). Furthermore, by considering a a regularized version of the function $\widetilde \varphi(x)=\left(\varphi (x)-|x-x_0|^4\right)^+$, we can also suppose that $\varphi (x)<d_\Omega(x)$, for every $x\in\Omega$ different from $x_0$.

 Let $y_0\in\partial\Omega$ be such that $d_\Omega (x_0)=|x_0-y_0|$, then 
 \begin{equation}\label{gradphi}
 \nabla \varphi (x_0)=\frac{x_0-y_0}{|x_0-y_0|}.
 \end{equation}
In order to prove this last equality, we first notice that, since $\phi$ is smooth, the inequality
 \begin{equation}\label{pol}
 \varphi (x)-\varphi(x_0)\le d_\Omega(x)-d_\Omega(x_0)\le |x-x_0|,
 \end{equation}
gives \(|\nabla \varphi (x_0)|\le 1\). Moreover, defining $x_t:=tx_0+(1-t)y_0$, we have 
 \[
 d_\Omega (x_t)=|x_t-y_0|=td_\Omega(x_0),
 \] 
 hence 
\begin{equation*}
-\nabla \varphi(x_0)\cdot \frac{x_0-y_0}{|x_0-y_0|}=\lim_{t\to 1} \frac{\varphi(x_t)-\varphi(x_0)}{|x_t-x_0|}\le\lim_{t\to 1}\frac{d_\Omega(x_t)-d_\Omega(x_0)}{|x_t-x_0|}= -1, 
\end{equation*}
which, together with \eqref{pol}, proves \eqref{gradphi}.

Let us now set $h:=\varphi(x_0)=d_\Omega(x_0)$ and choose a system of coordinates such that $x_0=0$ and the unit vector $e_\dim$ is parallel to \(x_0-y_0\). Since $\frac{\partial \varphi}{\partial x_\dim}\neq0$, by the Implicit Function Theorem, there is a $(\dim-1)$-dimensional ball $B^{\dim-1}_r\subset\R^d$ and a function $\phi\in C^\infty(B_r^{\dim-1})$  such that $\{\varphi=h\}$ is the graph of $\phi$ over $B_r^{\dim-1}$, i.e.
\begin{equation}
\{\varphi=h\}\cap \left(B_r^{\dim-1}\times(-r,r)\right)=\{x_\dim=\phi(x_1,\dots,x_{\dim-1})\}.
\end{equation}   
Since $d_\Omega\ge \varphi$ with equality only at \(x_0=0\), we  have
\begin{equation}\label{incllevel}
\{\varphi\ge h\}\subset \{d_\Omega\ge h\}\qquad \text{and}\qquad \{\varphi=h\}\cap \{d_\Omega=h\}=\{0\},
\end{equation}
which implies that  \(0\)  is a  (strict) local minimum of the function 
$$(x_1,\dots,x_{\dim-1})\mapsto x_1^2+\dots+x_{\dim-1}^2+(\phi-h)^2.$$
Hence $\frac{\partial\phi}{\partial x_1}(0)=\dots=\frac{\partial\phi}{\partial x_{\dim-1}}(0)=0$. On the other hand, since 
$$\varphi(x_1,\dots,x_{\dim-1},\phi(x_1,\dots,x_{\dim-1}))\equiv0,$$ 
we get, denoting with the subscripts the partial derivatives,
$$\varphi_j+\phi_j\varphi_\dim=0,$$
$$\varphi_{jj}+2\phi_j\varphi_{j\dim}+\phi_{jj}\varphi_\dim+\phi_j^2u_{\dim\dim}=0,$$
for each $j=1,\dots,\dim-1$, and thus we obtain $\varphi_{jj}(0)+\phi_{jj}(0)\varphi_\dim(0)=0$. By the contradiction assumption
$$0< \sum_{j=1}^\dim \varphi_{jj}(0)=\varphi_{\dim\dim}(0)-\varphi_\dim(0)\sum_{j=1}^{\dim-1}\phi_{jj}(0)\le -\varphi_\dim(0)\sum_{j=1}^{\dim-1}\phi_{jj}(0),$$
where the last inequality is due to 
\begin{equation*}
\varphi_{\dim\dim}(0)=\lim_{t\to0}\frac{\varphi(te_\dim)+\varphi(-te_\dim)-2\varphi(0)}{t^2}\le\lim_{t\to0}\frac{d_\Omega(te_\dim)+d_\Omega(-te_\dim)-2d_\Omega(0)}{t^2}\le0.
\end{equation*}
Since, by \eqref{gradphi}, we have $\varphi_\dim(0)=1$, we deduce that $\Delta\phi(0)< 0$. 

\begin{figure}[t]
\begin{center}
\includegraphics[scale=0.55]{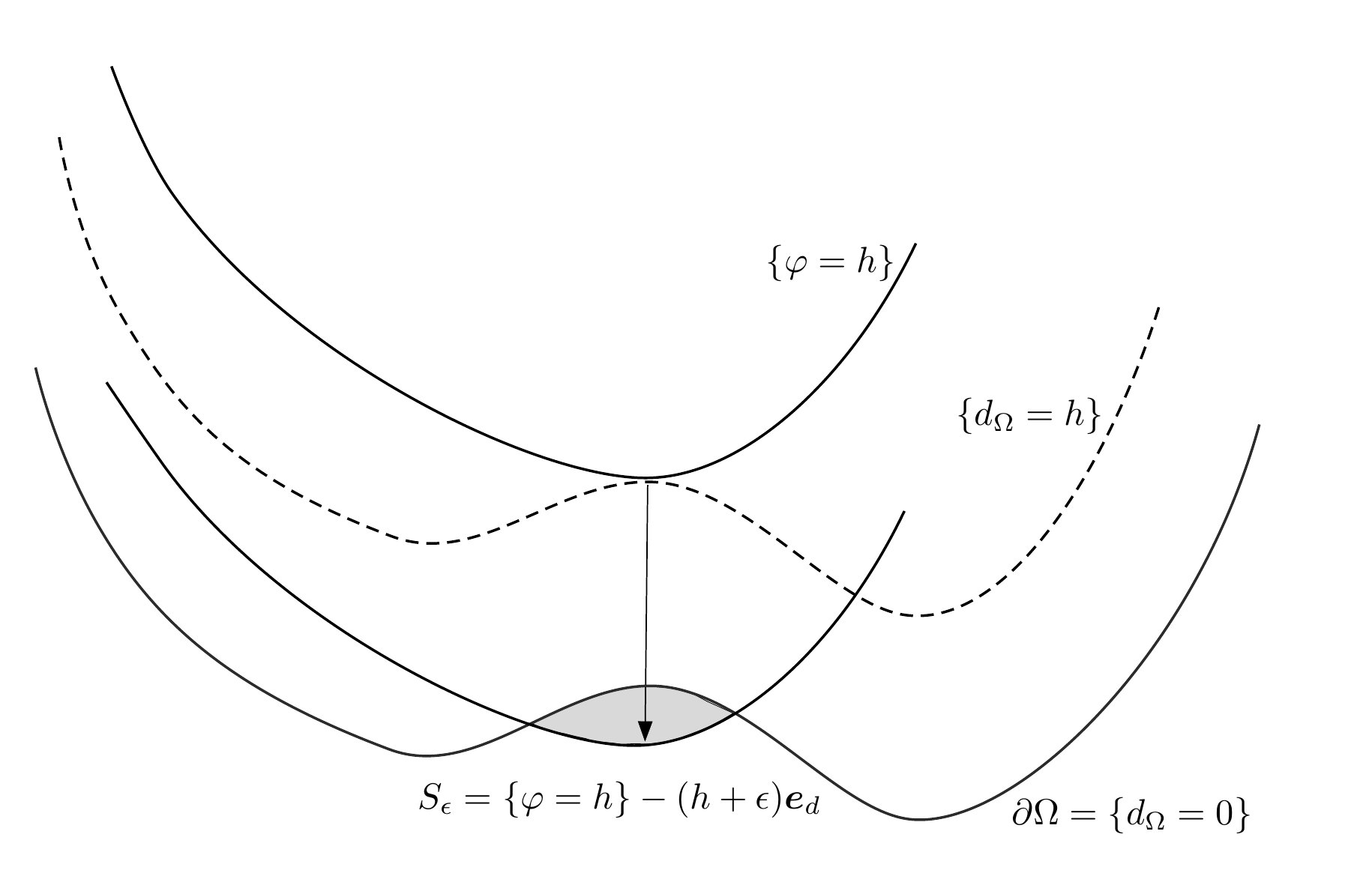}
\caption{Proof of Lemma \ref{dist}: applying the Divergence Theorem to the grey region, we obtain a contradiction to the minimality of \(\Omega\) if \(\Delta \varphi>0\).}
\label{fig}
\end{center}
\end{figure}

Let $d_S :T^+\to\R$ be the distance to the surface $S=\{x_\dim=\phi(x_1,\dots,x_{\dim-1})\}$. i.e. $d_S(x)=d(x,S)$, where $T$ is a tubular neighbourhood of $S$ and $T^+=\{x_\dim>\phi\}$. Then $d_S\in C^\infty(T^+\cup S)$,
\[
\frac{\partial d_S}{\partial x_d}(0)=1\qquad \text{and}\qquad \frac{\partial^2d_S}{\partial x_d^2}(0)=0.
\]
 Arguing as above, we see that $\Delta d_S(0)=-\Delta\phi(0)>0$ and so, $\Delta d_S>0$, in a neighbourhood of $0$ in $T^+\cup S$.
 
By \eqref{incllevel} we see  that for  $r$ small enough, there is some $\epsilon>0$ such that 
 \[
 \{h\le d_\Omega<h+\epsilon\}\cap\{\varphi\ge h\}\subset B_r. 
 \]
If we  define the set
  \[
  \Omega_\epsilon:=\Omega\cup\left(\{\varphi\ge h\}-(h+\epsilon)e_\dim\right),
  \]
then $\Omega_\epsilon\setminus\Omega\subset B_r(-(h+\epsilon)e_\dim)$. Denoting with \(d_\epsilon\) the distance from 
  $$S_\epsilon=\{\varphi=h\}-(h+\epsilon)e_\dim,$$
  we see that \(\Delta d_\epsilon >0\) in \(B_r(-(h+\epsilon)e_\dim)\), since $ d_\epsilon(x)=d_S(x+(h+\epsilon)e_\dim)$. Hence, by the Divergence Theorem, and recalling that on $S_\epsilon$, \(\nabla d_\epsilon=-\nu_{\Omega_\epsilon}\), where $\nu_{\Omega_\epsilon}$ is the exterior normal to \(\Omega_\epsilon\),  we have
\begin{equation}\label{dist7}
\begin{split}
0< \int_{\Omega_\epsilon\setminus\Omega}\Delta d_\epsilon\, dx&=-\int_{\Omega_\epsilon\cap\partial \Omega}\nabla d_\epsilon\cdot \nu_\Omega\,d\H^{\dim-1}-\int_{\Omega^c\cap\partial\Omega_\epsilon}\,d\H^{\dim-1}\\
&\le \H^{d-1}\big({\Omega_\epsilon\cap\partial \Omega}\big)-\H^{d-1}\big(\Omega^c\cap\partial\Omega_\epsilon\big),
\end{split}
\end{equation}
contradicting the perimeter minimality of \(\Omega\) with respect to outer variations (see Figure \ref{fig}).
\end{proof}

We are now in position to prove the Lipschitz continuity of $w_\Omega$ using \(d_\Omega\) as a barrier (see \cite[Chapter 14]{GT} for  similar proofs in the smooth case).
\begin{prop}\label{lip}
Suppose that the open set $\Omega\subset\R^\dim$ is a perimeter supersolution. Then the energy function $w_\Omega:\R^d\to\R$, defined as zero on $\Omega^c$, is Lipschitz continuous.
\end{prop}
\begin{proof}
For sake of simplicity, we set $w=w_\Omega$ and $\|\cdot\|_\infty=\|\cdot\|_{L^\infty(\Omega)}$. Let $c>2\|w\|_\infty^{1/2}$ and consider the function 
\begin{equation}\label{h}
h(t)=ct-t^2.
\end{equation}
 We claim that
 \begin{equation}\label{stima}
 w(x)\le h(d_\Omega(x)) \qquad \forall \, x\in \overline \Omega.
\end{equation}
Suppose this is not the case. Since both functions vanish on \(\partial \Omega\), there exists \(x_0 \in \Omega\) such that
\[
h(d_\Omega(x_0))-w(x_0)=\min_{\overline \Omega}\big\{h(d_\Omega)-w \big\},
\]
that is the function \(\varphi:=h^{-1}(w)\) touches \(d_\Omega\) from below. By our choice of \(c\) the function \(h\) is invertible on the range  of \(w\). Moreover, since \(w_\Omega(x_0)>0\), the inverse function is also smooth in a neighborhood of \(x_0\). By Lemma \ref{dist}, 
\[
\Delta \varphi(x_0)\le 0.
\]
Hence, the chain rule and the definition of \(h\), \eqref{h}, imply
\[
\Delta w(x_0)=\Delta \big(h\circ \varphi\big)(x_0)=h''(\varphi(x_0))|\nabla \varphi(x_0)|^2+h'(\varphi(x_0))\Delta \varphi(x_0)\le- 2|\nabla \varphi(x_0)|^2=-2,
\]
where we have also taken into account that, since \(\varphi\) touches \(d_\Omega\) from below at \(x_0\), equation \eqref{gradphi} implies the \(|\nabla \varphi(x_0)|=1\). Since \(-\Delta w =1\) the above equation cannot hold, hence \eqref{stima} holds true. Now equations \eqref{stima} and \eqref{h}, imply
\[
w(x)\le h(d_\Omega(x))\le c d_\Omega(x)\qquad \forall\, x\in \overline \Omega.
\]
Arguing as in the proof of Proposition \ref{holder}, we conclude that \(w\) is Lipschitz.
\end{proof}

In order to prove that any solution of \eqref{lbko} is a quasi-minimizer of the perimeter we need to use both the facts that it is a perimeter supersolution (see Lemma \ref{super}) and energy subsolution (see Proposition \ref{lag}). In the next lemma, we will obtain some local information on $\partial \Omega$ using an inner variation of $\Omega$. This lemma in combination with Lipschitz continuity of $\Omega$ will give the regularity of the boundary of $\Omega$ (see Theorem \ref{enper} below). The argument below is classical and similar results can be found in \cite{altcaf}, \cite{bulbk} and \cite{bubuve}. 
 
\begin{lemma}\label{sopra}
Let $\Omega\subset\R^d$ be an energy subsolution and let $w=w_\Omega$. Then for each $B_r(x_0)\subset\R^d$ and each $\widetilde\Omega\subset\Omega$ such that $\Omega\setminus\widetilde\Omega\subset B_r(x_0)$, we have the following inequality:
\begin{multline}\label{sopra1}
\frac{1}{2}\int_{B_r(x_0)}|\nabla w|^2\,dx+k(P(\Omega)-P(\widetilde\Omega))\\
\le \int_{B_r(x_0)}w\,dx+ C_d\left(r+\frac{\|w\|_{L^\infty(B_{2r}(x_0))}}{r}\right)\int_{\partial B_r(x_0)}w\,d\H^{d-1},
\end{multline}
\end{lemma}

\begin{proof}
Without loss of generality, we can suppose that $x_0=0$. We will denote with $B_r$ the ball of radius $r$ centered in $0$ and with $A_r$ the annulus $B_{2r}\setminus \overline{B_{r}}$.\\
Let $\psi:A_1\rightarrow\R^{+}$ be the solution of the equation:
$$\Delta \psi=0\ \hbox{on}\ A_1,\qquad\psi=0\ \hbox{on}\ \partial B_{1},\qquad\psi=1\ \hbox{on}\ \partial B_{2}.$$
We can also give the explicit form of $\psi$, but for our purposes, it is enough to know that $\psi$ is bounded and positive.\\
   With $\phi:A_1\rightarrow\R^{+}$ we denote the solution of the equation:
$$-\Delta \phi=1\ \hbox{on}\ A_1,\qquad\phi=0\ \hbox{on}\ \partial B_{1},\qquad\phi=0\ \hbox{on}\ \partial B_{2}.
$$
For an arbitrary $r>0$, $\alpha>0$ and $k>0$, we have that the solution of the equation

\begin{equation}\label{sopra4}
-\Delta v=1\ \hbox{on}\ A_r,\qquad v=0\ \hbox{on}\ \partial B_{1},\qquad v=\alpha\ \hbox{on}\ \partial B_{2},
\end{equation} 
is given by 
\begin{equation}\label{sopra5}
v(x)=r^2\phi(x/r)+\alpha\psi(x/r),
\end{equation}
and its gradient by 
\begin{equation}\label{sopra6}
\nabla v(x)=r\nabla\phi(x/r)+\frac{\alpha}{r}\nabla\psi(x/r).
\end{equation}

Consider the solution $v$ of the equation \ref{sopra4} with $\alpha=\|w\|_{L^{\infty}(B_{2r})}$. We work with the perturbation $w_r=w \1_{B_{2r}^c}+w\wedge v \1_{B_{2r}}$. Observe that, by the choice of $\alpha$,  $w_r\in H^1_0(\widetilde\Omega)$. Moreover, a simple computation gives
\[
|E(\widetilde \Omega)-E(\Omega)|\to 0 \quad\text{ uniformly as } r\to 0.
\]  
Being  $\Omega$  an energy subsolution, we get, for \(r\) small enough,
$$\frac{1}{2}\int_{\Omega}|\nabla w|^2\,dx-\int_\Omega w(x)\,dx+kP(\Omega)\le \frac{1}{2}\int_{\Omega}|\nabla w_r|^2\,dx-\int_{\Omega} w_r(x)\,dx+kP(\widetilde\Omega).$$
Since $w_r=0$ in $B_{r}$ and $w_r=w$ in $(B_{2r})^c$, we obtain with the same computations of Lemma \ref{bounded},
\begin{equation}
\begin{split}\frac{1}{2}\int_{B_{r}}|\nabla w|^2\,dx+k(P(\Omega)-P(\widetilde\Omega))&\le \frac{1}{2}\int_{A_r}|\nabla w_r|^2-|\nabla w|^2\,dx+\int_{B_{2r}}(w-w_r)\,dx\\
&\le -\int_{A_r}\nabla v\cdot\nabla ((w-v)\vee 0)\,dx+\int_{B_{2r}}(w-w_r)\,dx\\
&=\int_{\partial B_{r}}w\,\frac{\partial v}{\partial n}\,d\H^{d-1}-\int_{A_r}((w-v)\vee 0)\,dx+\int_{B_{2r}}(w-w_r)\,dx\\
&=\int_{\partial B_{r}}w\,\frac{\partial v}{\partial n}\,d\H^{d-1}+\int_{B_{r}}w\,dx\\
&\le \left(r\|\nabla\phi\|_{\infty}+\frac{\alpha}{r}\|\nabla\psi\|_{\infty}\right)\int_{\partial B_{r}}w\,d\H^{d-1}+\int_{B_{r}}w\,dx,
\end{split}\end{equation}
where the last inequality is due to \eqref{sopra6}. Recalling that $\alpha=\|w\|_{L^\infty(B_{2r})}$, we obtain \eqref{sopra1}.
\end{proof}

\begin{teo}\label{enper}
Let $\Omega\subset\R^d$ be a set of finite Lebesgue measure and finite perimeter. If $\Omega$ is an energy subsolution and a perimeter supersolution, then $\Omega$ is a bounded open set and its boundary is $C^{1,\alpha}$ for every $\alpha\in(0,1)$ outside a closed set of dimension \(d-8\).
\end{teo}
\begin{proof}First notice that, by Lemma \ref{bounded}, $\Omega$ is bounded. Moreover, since $\Omega$ is a perimeter supersolution, we can apply Proposition \ref{soblike} and Proposition \ref{lip}, obtaining that $\Omega$ is an open set and the energy function $w:=w_\Omega$ is Lipschitz.

We now divide the proof in two steps.

\medskip
\noindent
\(\bullet\) \textit{Step 1}.
 Let \(x_0\in \partial \Omega\) and let $B_r(x_0)$ be a ball of radius less than $1$. By Lemma \ref{sopra}, for each $\widetilde\Omega\subset \Omega $, such that $\widetilde\Omega\Delta\Omega\subset B_r(x_0)$, equation \eqref{sopra1} implies 
\begin{equation}\label{allafine}
\begin{split}
k(P(\Omega)-P(\widetilde\Omega))&\le \int_{B_r(x_0)}w\,dx+ C_d\left(r+\frac{\|w\|_{L^\infty(B_{2r}(x_0))}}{r}\right)\int_{\partial B_r(x_0)}w\,d\H^{d-1}\\
&\le C_d\Big(\|w\|_{L^\infty(B_{2r}(x_0))} r^{d-1}+\|w\|^2_{L^\infty(B_{2r}(x_0))} r^{d-2}\Big),
\end{split}
\end{equation}
where $C_d$ is a dimensional constant.
Now since \(w\) is Lipschitz and vanishes on \(\partial \Omega\), we have \(\|w\|_{L^\infty(B_{2r}(x_0))}\le C r\), hence equation \eqref{allafine}, implies 
\begin{equation}\label{almost}
P(\Omega,B_r(x_0))\le P(\widetilde\Omega,B_r(x_0))+Cr^{d},
\end{equation}
where $C$ depends on the dimension, the constant $k$ on \eqref{sub1} and the Lipschitz constant of \(w\) (which, in turn, depends only on the data of the problem). Moreover, by the perimeter subminimality, equation \eqref{almost} clearly holds true also for outer variations. Splitting every variation in an outer and inner variations (see for instance \cite[Proposition 1.2]{DPP}), we obtain
\begin{equation*}
P(\Omega,B_r)\le P(\widetilde\Omega,B_r)+Cr^{d},\qquad \forall\  \widetilde \Omega\Delta \Omega \subset B_r(x).
\end{equation*}
Hence \(\Omega\) is a \emph{almost-minimizer} for the perimeter in the sense of \cite{tam1,tam2} (see also \cite{alm}). From this it follows that \(\partial \Omega\) is a \(C^{1,\alpha}\) manifold, outside a closed \emph{singular set} \(\Sigma\) of dimension \((d-8)\), for every \(\alpha\in (0,1/2)\).

\medskip
\noindent
\(\bullet\) \textit{Step 2.} We want to improve the exponent of H\"older continuity of the normal of \(\partial \Omega\) in the regular (i.e. non-singular) points of the boundary. For this notice that, for every regular point \(x_0\in \partial \Omega\), there exists a radius \(r\) such that \(\partial\Omega\) can be represented by the graph of a \(C^1\) function \(\phi\) in \(B_r(x_0)\), that is, up to a rotation of coordinates
\[
\Omega \cap B_r(x_0)=\{x_d>\phi(x_1,\dots,x_{d-1})\}\cap B_r(x_0).
\]
For every  \(T\in C_c^1(B_r(x_0);\R^d)\) such that \(T\cdot \nu_\Omega<0\) and \(t\) is sufficiently small, we consider the local variation
\[
\Omega_t=\big({\rm Id}+tT\big)(\Omega) \subset \Omega.
\]
By the energy subminimality we obtain
\begin{equation}\label{comp}
k(P(\Omega)-P(\Omega_t))\le E(\Omega_t)-E(\Omega).
\end{equation}
Since \(T\) is supported in \(B_r\) and \(\partial \Omega\cap B_r\) is \(C^1\), we can perform the same computations as in \cite[Chapter 5]{hepi05}, to obtain that
\begin{equation}\label{derE}
E(\Omega_t)-E(\Omega)=-t\int_{\partial \Omega\cap B_r}\left |\frac{\partial w_\Omega}{\partial \nu}\right|^2\, T\cdot \nu_\Omega\, d\H^{d-1}+o(t).
\end{equation}
Moreover, see for instance \cite[Theorem 17.5]{maggi},
\begin{equation}\label{derP}
P(\Omega_t)=P(\Omega)+t\int_{\partial \Omega\cap B_r} {\rm div}_{\partial \Omega} T\, d\mathcal H^{d-1}+o(t)
\end{equation}
where \( {\rm div}_{\partial \Omega} T\) is the \emph{tangential divergence} of \(T\). Plugging \eqref{derE} and \eqref{derP} in \eqref{comp}, a standard computation (see \cite[Theorem 11.8]{maggi}), gives (in the distributional sense)
\[
{\rm div }\left(\frac {\nabla \phi}{\sqrt{1+|\nabla \phi|^2}}\right)\le \frac 1 k \left |\frac{\partial w_\Omega}{\partial \nu}\right|^2 \le C,
\]
where the last inequality is due to the Lipschitz continuity of \(w_\Omega\). Moreover applying \eqref{derP} to outer variations of \(\Omega\) (i.e. to variations such that \(T\cdot \nu_\Omega>0\)) we get
\[
{\rm div }\left(\frac {\nabla \phi}{\sqrt{1+|\nabla \phi|^2}}\right)\ge 0.
\]
In conclusion \(\phi\) is a \(C^1\) function satisfying
\[
{\rm div }\left(\frac {\nabla \phi}{\sqrt{1+|\nabla \phi|^2}}\right)\in L^\infty,
\]
and classical elliptic regularity gives \(\phi\in C^{1,\alpha}\), for every \(\alpha \in (0,1)\).
\end{proof}
 
\begin{oss}\label{analytic}In some particular cases the conclusion of the above Theorem can be strengthened. For instance, if one knows that, for a solution of \eqref{lbko}, the \(k\)th eigenvalue is simple, then one can write the Euler Lagrange equation for the minimum domain which reads, in a neighbourhood of a regular point, as
\[
{\rm div }\left(\frac {\nabla \phi}{\sqrt{1+|\nabla \phi|^2}}\right)=\left |\frac{\partial w_k}{\partial \nu}\right|^2,
\]
where \(w_k\) is the \(k\)th eigenfunction. 
In this case a standard bootstrap argument gives that the regular part of \(\partial \Omega\) is an analytic manifold. 
\end{oss} 
 
 \begin{figure}[htbp]
\begin{center}
\includegraphics[scale=1]{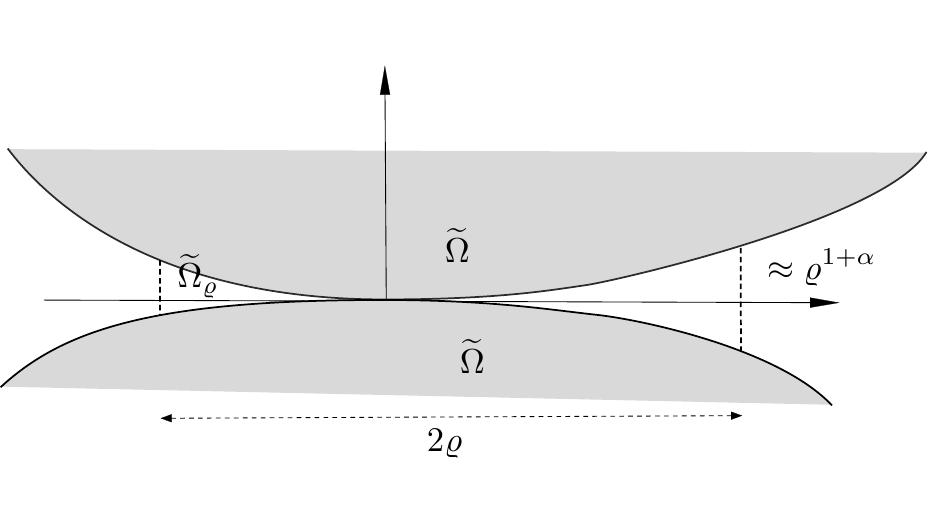}
\caption{The variation in the proof of Proposition \ref{conn}.}
\label{fig2}
\end{center}
\end{figure}

\begin{prop}\label{conn}
Let $f:\R^p\to\R$ be a function satisfying \((f1)\), \((f2)\) and \((f3)\) and let  $\Omega\subset\R^d$ be a bounded open set, solution of \eqref{lbfperint}. Then $\Omega$ is connected.
\end{prop}
\begin{proof}
We first prove the result in dimension \(d\le 7\), in which case the singular set of the boundary $\partial\Omega$ is empty. We first note that, since \(\Omega\) is a solution of \eqref{lbfperint}, it has a finite number (at most $max\{k_1,\dots,k_p\}$) of connected components. Suppose, by contradiction, that there are at least two connected components of $\Omega$. If we take one of them and translate it until it touches one of the others, then we obtain a set \(\widetilde \Omega\) which is still a solution of \eqref{lbfperint}. Using the regularity of the contact point for the two connected components, it is easy to construct an outer variation of \(\widetilde \Omega\) which decreases the perimeter (see Figure \ref{fig2}). In fact, assuming that the contact point is the origin, up to a rotation of the coordinate axes, we can find a small cylinder \(C_r\) and two \(C^{1,\alpha}\) functions \(g_1\) and \(g_2\) such that
\begin{equation}\label{norm}
g_1(0)=g_2(0)=|\nabla g_1(0)|=|\nabla g_2(0)|=0,
\end{equation}
and 
\[
\widetilde \Omega^c\cap C_r=\big\{ g_1(x_1,\dots,x_{d-1})\le x_d\le g_2(x_1,\dots, x_{d-1})\big\}\cap C_r.
\]
Now, for \(\varrho<r\), consider the set \(\widetilde \Omega_\varrho:=\widetilde\Omega\cup C_{\varrho}\supset\widetilde \Omega\). It is easy see that, thanks to \eqref{norm} and the \(C^{1,\alpha}\) regularity of \(g_1\) and \(g_2\),
\[
P(\widetilde \Omega_{\varrho})-P(\widetilde \Omega)\le C_\alpha \varrho^{d-1+\alpha}-C_d \varrho^{d-1}<0,
\]
for \(\varrho\) small enough, which contradicts the minimality of \(\widetilde \Omega\) \footnote{Another way to conclude is to notice that for \(\widetilde \Omega\) the origin is not a regular point, a contradiction with Theorem \ref{enper}.}.\\

We now consider the case $d\ge8$. In this case the singular set may be non-empty and so, in order to perform the operation described above, we need to be sure that the contact point is not singular. 

Suppose, by contradiction, that the optimal set $\Omega$ is disconnected, i.e. there exist two non-empty open sets $A,B\subset\Omega$ such that $A\cup B=\Omega$ and $A\cap B=\emptyset$. We have
\[
\partial A\cup \partial B \subset \partial \Omega=\partial^M\Omega, \quad(\footnote{We denote with \(\partial^M \Omega\) the essential boundary of \( \Omega\), i.e. the complement to the set of  density \(1\) points of \(\Omega\) and of \(\Omega^c\).})
\]
where the last inequality follows by classical density estimates. By Federer's criterion \cite[Theorem 16.2]{maggi}, \(A\) and \(B\) have finite perimeter.  Arguing as in  \cite[Theorem 2, Section 4]{ACMM}, we deduce that $P(\Omega)=P(A)+P(B)$. 

Since both $A$ and $B$ are bounded, there is some $x_0\in\R^d$ such that $\dist(A,x_0+B)>0$. Then the set $\Omega'=A\cup (x_0+B)$ is also a solution of \eqref{lbfperint}. Let $x\in\partial A$ and $y\in\partial (x_0+B)$ be such that $|x-y|=\dist(A,x_0+B)$. Since the ball with center $(x+y)/2$ and radius $|x-y|/2$ does not intersect $\Omega'$, we have that in both $x$ and $y$, $\Omega'$ satisfies the exterior ball condition. Hence both \(x\) and \(y\) are regular points\footnote{This can be easily seen, since any tangent cone at these points is contained in an half-space and hence it has to coincide with it, see \cite[Theorem 36.5]{sim}}.

Consider now the set $\Omega''=(-x+A)\cup(-y+x_0+B)$. It is a solution of \eqref{lbfperint} and has at least two connected components, which meet in a point which is regular for both of them. Reasoning as in the case $d\le 7$, we obtain a contradiction. 
\end{proof}

\section{Appendix: Proof of Theorem \ref{cc}}
We apply the concentration compactness principle from \cite{concom} to the sequence of characteristics functions $\1_{\Omega_n}$.  First notice that, being all the sets of finite measure, the isoperimetric inequality and the uniform bound on the perimeters ensure that
\[
\sup_{n\in \N}|\Omega_n|\le C.
\]
Let us assume that
\begin{equation}\label{limsup}
\limsup_{n\to \infty} |\Omega_n|>0.
\end{equation}
In this case, up to subsequences, we can assume that
\begin{equation*}
\l:=\lim_{n\to \infty }|\Omega_n|
\end{equation*}
exists and that \(\l\in (0,+\infty)\). Thanks to this  can rescale all our sets in such a way that $|\Omega_n|=1$ still maintaining a uniform bound on the perimeters.

As in \cite{concom} we have that, up to subsequences,  the family of concentration functions $Q_n:\R^+\to\R^+$, defined by 
$$Q_n(r)=\sup_{x\in\R^d}|B_r(x)\cap\Omega_n|,$$
is pointwise converging to some monotone increasing function $Q:\R^+\to\R^+$. We now consider three different cases:
\begin{enumerate}[(i)]
\item $\lim_{r\to\infty}Q(r)=1$. In this case, up to substitute  \(\Omega_n\)  with  \(\Omega_n+x_n\) for suitable \(x_n\in \R^d\),  we have that for every $\eps>0$ there is some $R>0$ such that
 \[
 \sup_{n\in \N}|\Omega_n\setminus B_R|\le\eps.
 \]
  Since the functions $w_{\Omega_n}$ are uniformly bounded in $L^\infty(\R^d)$  we infer  that for every $\eps>0$ there is some $R>0$ such that 
  \[ 
  \sup_{n\in \N} \int_{B_R^c}w_{\Omega_n}\,dx\le\eps.
  \]
By the compact inclusion $BV(\R^d)\hookrightarrow L^1_{\rm loc}(\R^d)$ and $H^1(\R^d)\hookrightarrow L^2_{\rm loc}(\R^d)$, we see that (up to  subsequences)  there  are a set $\Omega\subset\R^d$ of unit measure such that $\1_{\Omega_n}\to \1_\Omega$ in $L^1(\R^d)$ and  a function $w\in H^1(\R^d)$ such that $w_{\Omega_n}\to w_\Omega$ in $L^2(\R^d)$. Moreover, $w\ge0$ on $\R^d$ and $\{w>0\}\subset\Omega$.  By Proposition \ref{prop33} and the inequality $R_{\{w>0\}}\le R_\Omega$, we conclude that the \emph{compactness} $(i)$ holds. 

\item $\lim_{r\to\infty}Q(r)=\alpha\in(0,1)$. Let  \(\eps>0\), then  there exits  $r_\eps\ge 1/\eps$ such that for every $R\ge r_\eps$ we have $\alpha-\eps\le Q(R)\le\alpha$. By the  monotonicity of \(Q_n(r)\) and the pointwise convergence to \(Q(r)\) we can find  \(R_\eps \ge r_\eps+1/\eps\) and \(N_\eps\) such that  
\[
\alpha-2\eps\le Q_n(R)\le \alpha+\eps,\quad  \text{for every $n\ge N_\eps$ and every  $r_\eps\le R\le R_\eps$.}
\]
By the definition of $Q_n$ the above inequality  implies that there is a sequence $x_n\in\R^d$ such that 
$$\alpha-3\eps\le |\Omega_n\cap B_R(x_n)|\le \alpha+\eps\quad  \text{for every $n\ge N_\eps$ and every  $r_\eps\le R\le R_\eps$.}$$
Defining  
\[
A_n^\eps=\Omega_n\cap B_{r_\eps}(x_n) \quad \text{and}\quad B_n^\eps=\Omega_n\setminus\overline B_{R_\eps}(x_n),\]
 we see  that, thanks to the choice of \(R_\eps\), 
  \begin{equation}\label{pio1}
  \begin{array}{cc}
 d(A_n^\eps,B_n^\eps)\ge R_\eps-r_\eps\ge 1/\eps, \\
 \\
 \quad\big| |A_n|-\alpha|+\big | |B_n|-(1-\alpha)\big|\le 8\eps\quad \text{and}\quad  |\Omega_n\setminus(A_n^\eps\cup B_n^\eps)|\le 4\eps.
 \end{array}
\end{equation} 
Up to substitute $r_\eps$ and $R_\eps$ with some $\widetilde r_\eps\in(r_\eps,r_\eps+\sqrt{\eps})$ and $\widetilde R_\eps\in(R_\eps-\sqrt{\eps},R_\eps)$, we may suppose that 
$$\mathcal{H}^{d-1}(\partial B_{r_\eps}(x_n)\cap \Omega_n)+\mathcal{H}^{d-1}(\partial B_{R_\eps}(x_n)\cap \Omega_n)\le2\sqrt{\eps},$$ 
and, as a consequence,
\begin{equation}\label{pio2}
P(A_n^\eps\cup B_n^\eps)\le P(\Omega_n)+2\sqrt{\eps}.
\end{equation} 
It remains to estimate the difference $w_{\Omega_n}-w_{U_n}$, where $U_n:=A_n^\eps\cup B_n^\eps$. Let $\phi\in C^\infty_c(\R^d)$ be positive with support in $B_2$ and equal to $1$ on $B_1$. For $r>0$, consider the function $\phi_r(x)=\phi(x/r)$. Defining 
$$u_n^1:=\phi_{r_\eps/2}(\cdot-x_n)w_{\Omega_n}\in \widetilde H^1_0(A_n^\eps),\qquad u_n^2:=\left(1-\phi_{R_\eps}(\cdot-x_n)\right)w_{\Omega_n}\in\widetilde H^1_0(B_n^\eps),$$  
we have that 
\begin{equation}\label{master}
\|w_{\Omega_n}-u_n^1-u_n^2\|_{L^2}\le 4\eps\|w_{\Omega_n}\|_\infty,
\end{equation}
 where we noticed that  we may choose $r_\eps$ and $R_\eps$ still satisfying all the previous assumptions and such that
$$|(\Omega_n\cap B_{2R_\eps}(x_n))\setminus B_{r_\eps/2}(x_n)|\le 4\eps.$$
Moreover, there is some universal constant $C>0$ such that 
\begin{equation}\label{master2}
\int_{\R^d}|\nabla u_n^1|^2\,dx+\int_{\R^d}|\nabla u_n^2|^2\,dx- \int_{\R^d}|\nabla w_{\Omega_n}|^2\,dx\le\frac{C}{r_\eps}\le C\eps.
\end{equation}
where the last inequality follows by the choice of \(r_\eps\) we made at the beginning of the proof.

Since $U_n\subset\Omega_n$, we have that $w_{U_n}$ is the orthogonal projection of $w_{\Omega_n}$ on the space $\widetilde H^1_0(U_n)$ with respect to the \(H^1_0\) scalar product \footnote{To see this, just  notice that $w_{U_n}$ is the solution of the Euler-Lagrange equation associated to the functional $F:\widetilde H^1_0(U_n)\to\R$ defined by
 \[
 F(v)=\frac 1 2 \int|\nabla (w_{\Omega_n}-v)|^2\,dx=\frac 1 2 \int|\nabla w _{\Omega_n}|^2\,dx-\int v\,dx+\frac 1 2 \int |\nabla v |^2\,dx, 
  \]
where, in the last equality, we have taken into account the equation satisfied by \(w_{\Omega_n}\).}. Hence
\begin{equation}\label{pio3}
\begin{split}
&\int|\nabla(w_{\Omega_n}-w_{U_n})|^2\,dx\le \int|\nabla(w_{\Omega_n}-u_n^1-u_n^2)|^2\,dx\\
&=\int|\nabla w_{\Omega_n}|^2\,dx-2\int (u_n^1+u_n^2)\,dx+\int|\nabla u_n^1|^2\,dx+\int|\nabla u_n^2|^2\,dx\\
&=2\int \left(w_{\Omega_n}-u_n^1-u_n^2\right)\,dx+\int\left(|\nabla u_n^1|^2+|\nabla u_n^2|^2-|\nabla w_{\Omega_n}|^2\right)\,dx\le C\eps.
\end{split}
\end{equation}
where in the first and second equality we have taken into account  the equation satisfied by \(w_{\Omega_n}\) while in the last inequality we have used \eqref{master} and \eqref{master2}.
Sending \(\eps \to 0\) we see that \eqref{pio1} gives points (b) and (c) of the dichotomy case statement, \eqref{pio2} gives point (d), and \eqref{pio3} together with Proposition \ref{lemma36} gives (e). Since (a) is trivially true we obtain that in this case the \emph{dichotomy} $(ii)$ holds.

\item $\lim_{r\to\infty}Q(r)=0$. In this case the first part of the statement of  the vanishing case is clear, while for the second one we need some further considerations, similar to those in \cite[Proposition 3.5]{buc00}, based on the Lieb's Lemma (see \cite{lieb}). First notice  that by a truncation argument, it is enough to prove the statement in the case when all $\Omega_n$ are bounded sets. Since  $\|R_{\Omega_n}\|_{\mathcal{L}(L^2(\R^d))}=\lambda_1(\Omega_n)^{-1}$ a it is enough to prove that
 \begin{equation}\label{infinito}
 \lim_{n\to \infty} \lambda_1(\Omega_n)=+\infty.
  \end{equation}
  Let $\eps>0$ be fixed, $R>0$ be large enough and $N\in\N$ be such that for all $n\ge N$ and all $x\in\R$, we have $|\Omega_n\cap B_R(x)|\le\eps$. By the Lieb's Lemma, we have that there is some $x\in\R^d$ such that 
$$\lambda_1(\Omega_n\cap B_R(x))\le \lambda_1(\Omega_n)+\lambda_1(B_R(x)).$$
Using that $\lambda_1(B_r)=r^{-2}\lambda_1(B_1)$, for any $r>0$, and that the ball minimizes $\lambda_1$ among all sets of given measure, we have
\begin{equation}\label{vanishing}
\lambda_1(B_1)\eps^{-2/d}-R^{-2}\lambda_1(B_1)\le \lambda_1(\Omega_n).
\end{equation}
Since the left-hand side of \eqref{vanishing} goes to infinity as $\eps\to0$, we obtain \eqref{infinito}. 
\end{enumerate}  

Let us assume now that \eqref{limsup} does not hold. In this case, clearly \(\1_{\Omega_n}\to 0\) in \(L^1(\R^d)\), hence the same arguments of case (iii) imply that we are in the vanishing case.


\bigskip
{\small\noindent
Guido De Philippis:
Hausdorff Center for Mathematics, \\
Endenicher Allee 62, D-53115 Bonn-GERMANY\\
{\tt guido.dephilippis@hcm.uni-bonn.de}

\bigskip\noindent
Bozhidar Velichkov:
Scuola Normale Superiore di Pisa\\
Piazza dei Cavalieri 7, 56126 Pisa - ITALY\\
{\tt b.velichkov@sns.it}

\end{document}